%% file: HIC.tex
\documentclass[1p]{elsarticle_modified}
\usepackage{abbrmath_seonhwa} 
\usepackage{amsfonts}
\usepackage{amssymb}
\usepackage{amsmath}
\usepackage{amsthm}
\usepackage{amsbsy}
\usepackage{scalefnt}
\usepackage{caption,subcaption}
\usepackage[colorlinks]{hyperref}
\usepackage{color}
\usepackage{graphicx}
\usepackage{float}
\usepackage{setspace}
\usepackage{tikz-cd} 
\usepackage[numbers]{natbib}

\theoremstyle{definition}
\newtheorem{thm}{Theorem}[section]
\newtheorem{prop}[thm]{Proposition}

\newtheorem{conj}[thm]{Conjecture}
\newtheorem{defn}[thm]{Definition}
\newtheorem{exam}[thm]{Example}
\newtheorem{rmk}[thm]{Remark}
\theoremstyle{definition} 

\newenvironment{talign*}
{\csname align*\endcsname}
{\endalign}

\def\sl{\textrm{SL}(2,\mathbb{C})}
\def\psl{\textrm{PSL}(2,\mathbb{C})}
\def\Zbb{\mathbb{Z}}
\def\pmo{\{\pm 1\}}

\begin{document}
\begin{frontmatter}
	
	\title{On the Hikami-Inoue conjecture}

	\author{Jinseok Cho}
	\address{Department of Mathematics Education, Busan National University of Education}
	\ead{dol0425@bnue.ac.kr}
	
	\author{Seokbeom Yoon}
	\address{Department of Mathematical Sciences, Seoul National University}
	\ead{sbyoon15@snu.ac.kr}

	\author{Christian K. Zickert}
	\address{Department of Mathematics, University of Maryland, College Park}
	\ead{zickert@math.umd.edu}
	
	\begin{abstract}
		Given a braid presentation $D$ of a hyperbolic knot, Hikami and Inoue consider a system of polynomial equations arising from a sequence of cluster mutations determined by $D$. They show that any solution gives rise to shape parameters and thus determines a boundary-parabolic $\psl$-representation of the knot group. They conjecture the existence of a solution corresponding to the geometric representation. In this paper, we show that a boundary-parabolic representation $\rho$ arises from a solution if and only if the length of $D$ modulo $2$ equals the obstruction to lifting $\rho$ to a boundary-parabolic $\sl$-representation (as an element in $\Zbb_2$). In particular, the Hikami-Inoue conjecture holds if and only if the length of $D$ is odd. This can always be achieved by adding a kink to the braid if necessary. We also explicitly construct the solution corresponding to a boundary-parabolic representation given in the Wirtinger presentation of the knot group.
	\end{abstract}

\end{frontmatter}


\section{Introduction}

Let $D$ be a braid of length $n$ and width $m$. Hikami and Inoue \cite{hikami2015braids} considered $n+1$ cluster variables $\mathbf{x}^1,\mathbf{x}^2,\cdots, \mathbf{x}^{n+1}$, each of which consists of $3m+1$ variables, and  related two consecutive cluster variables $\mathbf{x}^{i}$ and $\mathbf{x}^{i+1}$ ($1 \leq i \leq n$) by an operator arising from cluster mutations. Precisely, if $D$ has a braid group presentation $\sigma_{k_1}^{\epsilon_1}\sigma_{k_2}^{\epsilon_2}\cdots\sigma_{k_n}^{\epsilon_n}$, where $\sigma_{k_i}$ denotes the standard generator of the $m$-braid group and $\epsilon_i=\pm1$, then we have 
$$\mathbf{x}^2=R_{k_1}^{\epsilon_1}(\mathbf{x}^1),\ \mathbf{x}^3 = R_{k_2}^{\epsilon_2}(\mathbf{x}^2),\mkern4mu\cdots,\ \mathbf{x}^{n+1}=R_{k_n}^{\epsilon_n}(\mathbf{x}^{n})$$ where $R^\pm_k$ is the operator given by
\begin{equation*}
R^\pm_k (x_1,\cdots,x_{3m+1}) = \left(x_1,\cdots,x_{3k-3},\  R^\pm(x_{3k-2},\cdots,x_{3k+4}),\  x_{3k+5},\cdots,x_{3m+1}\right).
\end{equation*}  We refer to the equations \eqref{eqn:rp} and \eqref{eqn:rm}
for the definition of $R^\pm$. See also \cite[\S 2.2]{hikami2015braids}.

\begin{defn}\label{defn:sol} The initial cluster variable $\mathbf{x}^1\in \Cbb^{3m+1}$ is called a \emph{solution} if $\mathbf{x}^1=\mathbf{x}^{n+1}$. 
\end{defn} 

Recall that the space $S^3 \setminus (K\cup\{p, q\})$ admits a decomposition into ideal octahedra, where $K$ is the knot represented by $D$ and $p\neq q\in S^3$ are two points not in $K$. See, for instance, \cite{thurston1999hyperbolic}, \cite{weeks2005computation}, or Section \ref{sec:oct}. Dividing each ideal octahedron into ideal tetrahedra as in Figure 4 of \cite{hikami2015braids},
Hikami and Inoue proved that a non-degenerate solution (see Definition \ref{defn:nond}) determines the shape parameter of each ideal tetrahedron so that these tetrahedra satisfy the gluing equations and completeness condition. In particular, we obtain  a boundary-parabolic representation \[\rho_{\mathbf{x}^1}:\pi_1(S^3\setminus K)=\pi_1(S^3 \setminus (K\cup\{p, q\}))\rightarrow \psl\] up to conjugation from a non-degenerate solution $\mathbf{x}^1$.

\begin{conj}\cite[Conjecture 3.2]{hikami2015braids} \label{conj:hi} Let $D$ be a braid presentation of a hyperbolic knot $K$. Then there exists a non-degenerate solution $\mathbf{x}^1$  such that the induced representation $\rho_{\mathbf{x}^1}$ is geometric, i.e., discrete and faithful.
\end{conj}


\begin{rmk} \label{rmk:nond} In this paper, we shall use a different subdivision of an octahedron from \cite{hikami2015braids} (see Figure \ref{fig:fiveterm}). A non-degenerate solution, implying the non-degeneracy of the ideal tetrahedra, thus requires a slightly different condition (see Definition \ref{defn:nond}) from \cite{hikami2015braids}.
Henceforth, by a non-degenerate solution we mean a solution that satisfies the condition in Definition \ref{defn:nond}.
	We stress that this change of an ideal triangulation is essential for the existence of a non-degenerate solution (see Remark~\ref{rmk:essen}).
\end{rmk}

The main purpose of this paper is to analyze the conjecture. In particular, we prove the following, which is a consequence of the more general results Theorems~\ref{thm:int1} and \ref{thm:int2} below.
\begin{thm} \label{thm:ans}
	Conjecture \ref{conj:hi} holds if and only if the length of the braid is odd.
\end{thm}
Note that one can always make the braid length odd by adding a kink if necessary. 

\subsection{Main results}
Let $M$ be a compact $3$-manifold with non-empty boundary and $G$ be either $\psl$ or $\sl$.
Recall that a representation $\rho:\pi_1(M)\to G$ is \emph{boundary-parabolic} if it maps peripheral subgroups to conjugates of the subgroup $P$ of $G$ consisting of upper triangular matrices with ones on the diagonal. We shall sometimes call such a representation $\rho$ a \emph{$(G,P)$-representation}.

A representation $\pi_1(M)\to\psl$ may or may not lift to $\sl$ and the obstruction to lifting is a class in $H^2(M;\{\pm 1\})$. Also, a boundary-parabolic $\psl$-representation may lift to an $\sl$-representation which is not boundary-parabolic. The obstruction to lifting a boundary-parabolic $\psl$-representation $\rho$ to a boundary-parabolic $\sl$-representation is a class in $H^2(M,\partial M;\pmo)$ called the \emph{obstruction class of $\rho$}~\cite{garoufalidis2015complex,garoufalidis2015ptolemy}. Note that the image of this class in $H^2(M;\{\pm1\})$ is the obstruction to lifting $\rho$ to $\sl$. If $M=S^3\setminus \nu(K)$, where $\nu(K)$ denotes a small open regular neighborhood of a knot $K$,  then we have $H^2(M,\partial M ;\pmo) \simeq \pmo$. Therefore, the obstruction class of a boundary-parabolic $\psl$-represntation $\rho : \pi_1(M)\rightarrow \psl$ can be viewed as an element of $\pmo$. 
\begin{thm} \label{thm:int1}  Let $D$ be a braid of a knot $K$ (not necessarily hyperbolic). Then the obstruction class of $\rho_{\mathbf{x}^1}$ induced from a non-degenerate solution $\mathbf{x}^1$ is  $(-1)^n$ where $n$ is the length of $D$.
\end{thm}
 	
The obstruction class of the geometric representation of a hyperbolic knot is non-trivial. This follows from the fact that any lift of the geometric representation maps a longitude to an element with trace $-2$ (see e.g.~\cite{calegari2006real}, \cite[\S 3.2] {menal2012twisted} and also Proposition~\ref{prop:knot} below). Hence, Theorem \ref{thm:int1} shows that having odd braid length is necessary for Conjecture~\ref{conj:hi} to hold. The fact that this is also sufficient follows from the result below, which is proved in Section \ref{sec:comp}.

\begin{thm} \label{thm:int2} Let $D$ be a braid of a knot $K$ (not necessarily hyperbolic) and  $\rho : \pi_1(S^3\setminus K)\rightarrow \psl$ be a non-trivial boundary-parabolic representation. If the  obstruction class of $\rho$ is $(-1)^n$, where $n$ is the length of $D$, then there exists a non-degenerate solution $\mathbf{x}^1$ such that the induced representation $\rho_{\mathbf{x}^1}$ coincides with $\rho$ up to conjugation.	
\end{thm}
We remark that the solution can be constructed explicitly when $\rho$ is given using the Wirtinger presentation of the knot group. This uses techniques developed in~\cite{cho2016optimistic}.

\subsection{Organization of the paper}
In Section \ref{sec:coy}, we recall the notion of Ptolemy coordinates with obstruction class. In Section \ref{sec:main}, we give a short review on the Hikami-Inoue cluster variables and clarify the relation between these cluster variables and Ptolemy assignments by constructing a particular obstruction cocycle (Section \ref{sec:hicocycle}). This gives a proof of Theorem \ref{thm:int1}. In  Section~\ref{sec:dec}, we prove Theorem \ref{thm:int2} and present an explicit way to compute a solution when a boundary-parabolic representation is given in the Wirtinger presentation of the knot group.

\subsection{Acknowledgements} Christian Zickert was supported by NSF grant DMS-1711405.

\section{Ptolemy varieties with obstruction class} \label{sec:coy}

Let $M$ be an oriented compact $3$-manifold with non-empty boundary. 
We fix an ideal triangulation $\Tcal$ of the interior of $M$. This endows $M$ with a decomposition into truncated tetrahedra whose triangular faces triangulate $\partial M$ (see Figure \ref{fig:trunctaion}). We denote by $M^i$ or $\partial M^i$ the set of the oriented $i$-cells (unoriented when $i=0$). We call an edge of $\partial M$ a \emph{short-edge} and call an edge of $M$ not in $\partial M$ a \emph{long-edge}. For an oriented 1-cell $e$, we let $-e$ denote $e$ with its opposite orientation.
\subsection{Obstruction classes}

For a group $G$ the set $C^i(M;G)$ of all set maps from $M^i$ to $G$ forms a group with the operation naturally induced from $G$. 
We call $\sigma \in C^1(M;G)$ a \emph{$G$-cocycle} if it satisfies
\begin{enumerate}[(i)]
	\item $\sigma(e) \sigma(-e)=1$  for all $e \in M^1$;
	\item $\sigma(e_1) \sigma(e_2)\cdots \sigma(e_m) =1$ for each face $f$ of $M$ where $e_1,\cdots,e_m$ are the boundary edges of the face in the cyclic order determined by a choice of orientation of $f$.\label{cocycle}
\end{enumerate}
The set  $Z^1(M;G)$ of all $G$-cocycles admits a
$C^0(M;G)$-action defined as follows.
\begin{equation*}
Z^1(M; G) \times C^0(M; G) \rightarrow Z^1(M; G), \quad (\sigma,\tau)\mapsto \sigma \boldsymbol{\cdot} \tau
\end{equation*}
where $\sigma \boldsymbol{\cdot} \tau : M^1 \rightarrow G$ is given by $(\sigma \boldsymbol{\cdot} \tau)(e) = \tau(v)^{-1} \sigma(e) \tau(w)$ for $e \in M^1$, where $v$ and $w$ are the initial and terminal vertices of $e$, respectively. 
The following fact is well-known (see e.g.~\cite{zickert2009volume,neumann2004extended}). 

\begin{prop} \label{prop:bij}
	The orbit space $H^1(M;G):=Z^1(M;G)/C^0(M;G)$ has a natural bijection with the set of all conjugacy classes of representations $\rho : \pi_1(M)\rightarrow G$.
\end{prop}

Note that if $G$ is abelian, $H^1(M;G)$ is canonically isomorphic to the usual cellular cohomology group with the coefficient $G$.

Let $G$ be either $\sl$ or $\psl$ and $P$ be the subgroup of $G$ consisting of the upper triangular matrices with ones in the diagonal. We let $C^i(M,\partial M; G,P)$ be the subset of $C^i(M;G)$ consisting of elements $\sigma \in C^i(M;G)$ satisfying $\sigma(x) \in P$ for all $x \in \partial M^i$, and let $Z^1(M,\partial M ;G,P) := Z^1(M;G) \cap \mkern1mu C^1(M,\partial M;G,P)$. 
An element of $Z^1(M,\partial M ;G,P)$ is called a \emph{$(G,P)$-cocycle}. One can easily check (see e.g.~\cite{zickert2009volume}) that every $(G,P)$-representation can be represented by a $(G,P)$-cocycle. In fact, $H^1(M,\partial M;G,P) :=Z^1(M, \partial M;G,P)/C^0(M,\partial M;G,P)$ is in natural bijection with the set of (conjugacy classes of) so-called \emph{decorated $(G,P)$-representations} (see e.g.~\cite{zickert2009volume,garoufalidis2015complex}), but we shall not need this here.

From the short exact sequence of groups $1 \rightarrow \pmo \rightarrow \sl \rightarrow \psl \rightarrow 1$, we obtain exact sequences (the standard proof of exactness still works in low degree even though the terms are only sets, not groups)
\begin{equation*}
H^1(M;\sl)\rightarrow H^1(M;\psl) \rightarrow H^2(M;\pmo) \ \textrm{ and }
\end{equation*}
\begin{equation*}
H^1(M,\partial M;\sl,P)\rightarrow H^1(M,\partial M;\psl,P)\overset{\delta}{\rightarrow} H^2(M,\partial M;\pmo).
\end{equation*} 
In particular, the latter sequence tells us that a $(\psl,P)$-representation $\rho$ admits a $(\sl,P)$-lifting if and only if $\delta(\rho) \in H^2(M,\partial M;\pmo)$ vanishes, where $\rho$ is viewed as a $(\psl,P)$-cocycle. The element $\delta(\rho)$ is called the \emph{obstruction class} of $\rho$. Note that it does not depend on the choice of a $(\psl,P)$-cocycle representing $\rho$. Recall that we have the long exact sequence
$$ H^1(M;\pmo) \rightarrow H^1(\partial M;\pmo) \rightarrow H^2(M,\partial M;\pmo) \rightarrow H^2(M;\pmo).$$ 
It thus follows that if $\rho$ lifts to $\sl$ (e.g.~ if $H^2(M;\pmo)=0$), then the obstruction class of $\rho$ in $H^2(M,\partial M;\pmo)$ can be viewed as an element of $\textrm{Coker}(H^1(M;\pmo) \rightarrow H^1(\partial M;\pmo))$. In particular, if $M$ is a knot exterior in $S^3$, the obstruction class of $\rho$ is determined by the lift of the longitude. More precisely, the following holds. 


\begin{prop} \label{prop:knot} Let $K \subset S^3$ be a knot  and $M$ be the knot exterior.   Then the obstruction class of $(\psl,P)$-representation $\rho$ (as an element of $H^2(M,\partial M; \pmo) \simeq \{\pm1\}$) coincides with half of $\textrm{tr}(\widetilde{\rho}(\lambda))$ where $\widetilde{\rho} : \pi_1(M)\rightarrow \sl$ is any lift of $\rho$ and $\lambda$ is the canonical longitude of $K$.
\end{prop}
\begin{proof}
	Considering any Wirtinger presenation of $\pi_1(M)$, it is easy to check that $\rho$ has only two lifts $\widetilde{\rho}_+$ and $\widetilde{\rho}_- : \pi_1(M)\rightarrow \sl$ such that $\textrm{tr}(\widetilde{\rho}_+(\mu))=2$ and $\textrm{tr}(\widetilde{\rho}_-(\mu))=-2$, respectively, where $\mu$ is a merdian of $K$. Since $\pi_1(\partial M)$ is an abelian group generated by $\mu$ and $\lambda$, $\rho$ admits a $(\sl,P)$-lifting if and only if $\textrm{tr}(\widetilde{\rho}_+(\lambda))=2$. Therefore, by definition, the obstruction class of $\rho$ coincides with half of $\textrm{tr}(\widetilde{\rho}_+(\lambda))$. On the other hand, the canonical longitude $\lambda$ is in the commutator subgroup of $\pi_1(M)$ and thus it should be expressed in Wirtinger generators of even length. Therefore, we have $\widetilde{\rho}_+(\lambda)=\widetilde{\rho}_-(\lambda)$.
\end{proof}

\subsection{Ptolemy varieties} \label{sec:ptol}
Recall that $\Tcal$ is an ideal triangulation of the interior of a compact manifold $M$. 
We denote by $\Tcal^1$ the set of the oriented 1-cells. We shall often identify each $e \in \Tcal^1$ with a long-edge of $M$ in a natural way.

The third author with Garoufalidis and Thurston \cite{garoufalidis2015complex} (see also \cite{zickert2009volume}) gave an efficient parametrization of $(\psl,P)$-representations with a given obstruction class. Precisely, for $\sigma\in Z^2(M,\partial M;\pmo)$ they defined the \emph{Ptolemy variety $P^{\sigma}(\Tcal)$ with the obstruction cocycle $\sigma$}  by the set of all set maps $c : \Tcal^1 \rightarrow \Cbb \setminus\{0\}$ satisfying $-c(e)=c(-e)$ for all $e \in \Tcal^1$ and
\begin{equation} \label{eqn:obs}
\sigma_2\,	c(l_{02})c(l_{13})=\sigma_3\, c(l_{03})c(l_{12})+  \sigma_1\,	 c(l_{01}) c(l_{23})
\end{equation}  for each ideal tetrahedron $\Delta$ (with vertices $\{0,1,2,3\}$) of $\Tcal$, where $l_{ij}$ is the oriented edge of $\Delta$ going from vertex $i$ to vertex $j$,  and $\sigma_i$ is the $\sigma$-value on the hexagonal face opposite to the vertex $i$. See Figure \ref{fig:trunctaion}. We call an element $c \in P^\sigma(\Tcal)$ a \emph{Ptolemy assignment}.
\begin{figure}[!h]
	\centering
	\scalebox{1}{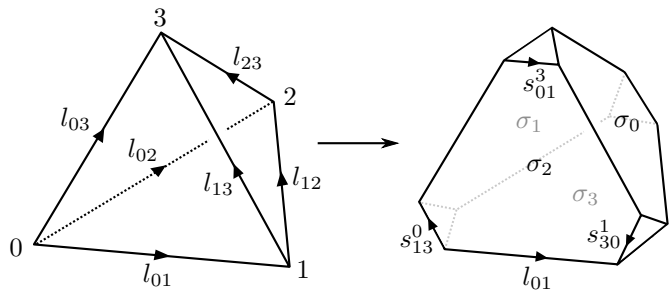}
	\caption{A truncated tetrahedron.}
	\label{fig:trunctaion}
\end{figure}

A Ptolemy assignment $c \in P^\sigma(\Tcal)$ corresponds to a $(\psl,P)$-cocycle, $\Phi_c$, such that
$\delta(\Phi_c)=[\sigma] \in H^2(M,\partial M ;\pmo)$. It thus \emph{induces} a $(\psl,P)$-representation $\rho_c : \pi_1(M)\rightarrow \psl$ up to conjugation whose obstruction class is $[\sigma]$. Note that $\Phi_c$ is explicitly expressed by $c \in P^\sigma(\Tcal)$ as follows. 
\begin{equation*} 
\Phi_c(l_{ij}) =\pm \begin{pmatrix}
0 & -c(l_{ij})^{-1} \\ c(l_{ij}) & 0
\end{pmatrix}
,\quad
\Phi_c(s^k_{ij}) =\pm \begin{pmatrix}
1 & -\sigma_l \dfrac{c(l_{ji})}{c(l_{ik}) c(l_{kj})} \\ 0 & 1
\end{pmatrix} 
\end{equation*}  for Figure \ref{fig:trunctaion}. Here $\{i,j,k,l\}=\{0,1,2,3\}$ and $s^k_{ij} \in \partial M^1$ is the edge contained in the face $[i,j,k]$ and parallel to $l_{ij}$ (see Figure \ref{fig:trunctaion}).
The cocycle condition~\eqref{cocycle} is then automatically satisfied for the hexagonal faces, and the Ptolemy relation~(\ref{eqn:obs}) ensures that it is also satisfied for the triangular faces. We refer \cite[\S 9]{garoufalidis2015complex} for details.

Now let us consider the map
\begin{equation}\label{eqn:d}
d : Z^1(\partial M ;\pmo) \rightarrow Z^2(M,\partial M;\pmo), \ \epsilon \mapsto d(\epsilon)
\end{equation}
defined by $d(\epsilon)$-value on a face of $M$ by multiplying $\epsilon$-values of all edges of $\partial M$ that are contained in the face. We note that it induces the usual map $H^1(\partial M;\pmo )\rightarrow H^2(M,\partial M;\pmo)$.
\begin{prop} \label{prop:lift}
	Let $\epsilon \in Z^1(\partial M ;\pmo)$. Then any $(\psl,P)$-representation $\rho_c$ induced from $c \in P^{d(\epsilon)}(\Tcal)$ admits a lift $\widetilde{\rho}_c :\pi_1(M)\rightarrow \sl$ such that
	\begin{equation*}
	\widetilde{\rho}_c(\gamma) =\begin{pmatrix}
	\overline{\epsilon}(\gamma) & * \\ 0 & \overline{\epsilon}(\gamma) 
	\end{pmatrix} 
	\end{equation*}for all $\gamma \in \pi_1(\partial M)$ up to conjugation, where $\overline{\epsilon} : \pi_1(\partial M)\rightarrow \pmo$ is the homomorphism induced from  $\epsilon$. 
\end{prop}
\begin{proof} 
	We may choose a lift $\widetilde{\Phi}_c \in C^1(M;\sl)$ of $\Phi_c$ such that
	\begin{equation*}
	\widetilde{\Phi}_c(l) = \begin{pmatrix}
	0 & -c(l)^{-1} \\ c(l) & 0
	\end{pmatrix}	
	\textrm{ and }
	\widetilde{\Phi}_c(s) = \begin{pmatrix}
	1 & \ast \\ 0 & 1
	\end{pmatrix}
	\end{equation*} for all $l \in M^1 \setminus \partial M^1$ and $s \in \partial M^1$. One can check that $\widetilde{\Phi}_c$ satisfies the cocycle condition for every triangular face of $M$ (but may not for all faces).
	Let $\widetilde{\epsilon} \in C^1(M;\pmo)$ be the trivial extension of $\epsilon$, i.e., $\widetilde{\epsilon}(e) := \epsilon(e)$ if $e \in \partial M^1$ and  $\widetilde{\epsilon}(e):=1$, otherwise.
	Then by definition $\widetilde{\epsilon} \cdot \widetilde{\Phi}_c : M^1 \rightarrow \sl$ is a cocycle satisfying
	$$(\widetilde{\epsilon} \cdot \widetilde{\Phi}_c)(e)= \begin{pmatrix}
	\epsilon(e) & * \\ 0 & \epsilon(e)
	\end{pmatrix}$$ for all $e \in \partial M$.
	The proposition follows by letting $\widetilde{\rho}_c :\pi_1(M)\rightarrow \sl$ be a representation induced from $\widetilde{\epsilon} \cdot \widetilde{\Phi}_c$. 
\end{proof}

Combining Propositions \ref{prop:knot} and \ref{prop:lift}, we obtain the following.
\begin{thm} \label{thm:key} Let $K \subset S^3$ be a knot and $M=S^3 \setminus \nu(K)$. Then any $(\psl,P)$-representation $\rho_c$ induced from $c\in P^{d(\epsilon)}(\Tcal)$ has the obstruction class $\overline{\epsilon}(\lambda) \in \pmo \simeq H^2(M,\partial M;\pmo)$, where $\overline{\epsilon} :\pi_1(\partial M)\rightarrow \pmo$ is  the homomorphism induced from $\epsilon \in Z^1(\partial M ; \pmo)$ and $\lambda$ is the canonical longitude of $K$.
\end{thm}

\section{The Hikami-Inoue cluster variables} \label{sec:main}

\subsection{The octahedral decomposition of a knot complement with two points removed}\label{sec:oct}
Let $K \subset S^3$ be a knot and let $\nu(K\cup\{p,q\})$ denote a tubular neigborhood of the union of $K$ with two points $p\neq q\in S^3$ not in $K$. Whenever we choose a knot diagram of $K$, we have a decomposition of the space $M=S^3\setminus \nu \mkern1mu(K\cup \{p,q\})$ into blocks each of which is a cube with two cylinders (whose core is the knot) removed. See Figure \ref{fig:block}. Note that $M$ is a $3$-manifold with 3 boundary components (two
spheres and a torus) whose interior is homeomorphic to $S^3\setminus(K\cup\{p,q\})$. Now consider two quadrilaterals $Q_1$ and $Q_2$ in each block as in Figure \ref{fig:block} and collapse them horizontally so that their vertical edges are respectively identified. We call the resulting object a \emph{pinched block}. 
		 
		\begin{figure}[!h]
			\centering
			\scalebox{1}{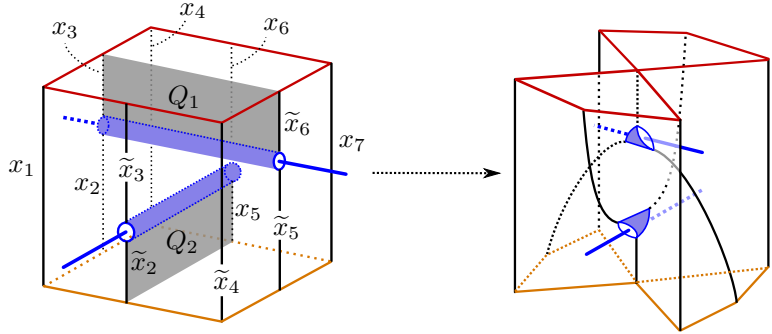}
			\caption{A pinched block}
			\label{fig:block}
		\end{figure} 
	
		On the other hand, a pinched block can also be obtained from a truncated octahedron by identifying two pairs of edges as in Figure \ref{fig:fiveterm}~(right). 
		Therefore, one can obtain $M$ by gluing truncated octahedra, and it thus follows that the interior of $M$ can be decomposed into ideal octahedra  (one per crossing). We denote  by $\Ocal$ this octahedral decomposition of $S^3\setminus(K\cup\{p,q\})$. It is due to Dylan Thuston~\cite{thurston1999hyperbolic} (see also \cite{weeks2005computation}).


		\begin{figure}[!h]
			\centering
			\scalebox{1}{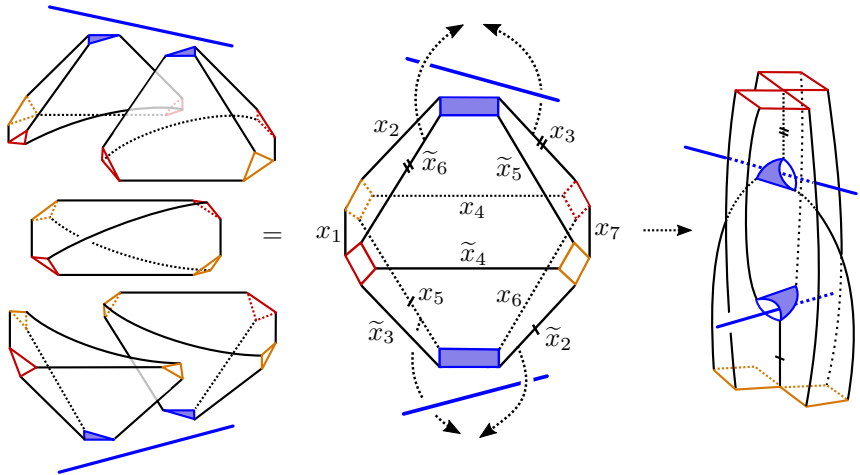}
			\caption{A truncated octahedron}
			\label{fig:fiveterm}
		\end{figure}
	
	\subsection{The Hikami-Inoue cluster variables}\label{sec:octa}
	The edges of an ideal octahedron (as in Figure~\ref{fig:fiveterm}) correspond to  vertical edges of a block as in Figure~\ref{fig:block}~(left). We label these edges by $x_1,\cdots,x_7,\widetilde{x}_1,\cdots,\widetilde{x}_7$ as in Figure \ref{fig:tempt} with the obvious relations $x_1=\widetilde{x}_1$ and $x_7=\widetilde{x}_7$. As indicated in Figure~\ref{fig:tempt}~(left) we shall regard the edges $x_i$ as being above a crossing, and the edges $\widetilde{x}_i$ as being below the crossing.
	\begin{figure}[!h]
	\centering
	\scalebox{1}{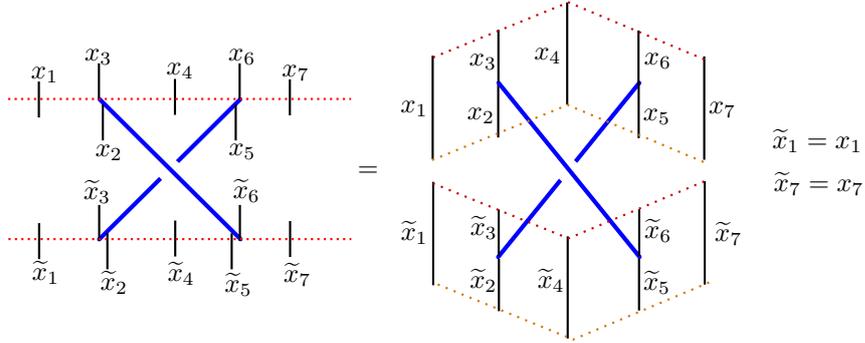}
	\caption{Edges of an octahedron at a crossing}
	\label{fig:tempt}
	\end{figure} 
	
Assigning a complex-valued variable to each of the edges $x_1,\cdots,x_7$, $\widetilde{x}_1,\cdots,\widetilde{x}_7$ with the same label as the edge itself, Hikami and Inoue considered the equation \[(\widetilde{x}_1,\cdots,\widetilde{x}_7)=R^\pm(x_1,\cdots,x_7)\] (see \cite[\S 2.2]{hikami2014cluster}) where $R^\pm$ is a certain operator defined by rational polynomial equations. As we shall see in Section~\ref{sec:hicocycle}, these equations are equivalent to Ptolemy relations for a particular obstruction cocycle. 

	Now suppose that the knot diagram $D$ is given by a braid with presentation $\sigma_{k_1}^{\epsilon_1}\cdots\sigma_{k_n}^{\epsilon_n}$. Here $\sigma_{k_i}$ denotes the standard generator of the $m$-braid group and $\epsilon_i\in\pmo$. Similar to the edge-labeling described in the previous paragraph, we label the oriented edges of the octahedral decomposition $\Ocal$ as follows.
	\begin{enumerate}
	\item Draw $n+1$ imaginary horizontal lines on the braid $D$ so that there is only one crossing between two consecutive lines (see Figures  \ref{fig:ocal} and \ref{fig:braid}). 
	\item As in Figure \ref{fig:tempt}~(left), whenever a horizontal line meets  $D$ there are two corresponding edges, and whenever a horizontal line meets a region of (the closure of) $D$, there is one corresponding edge. Since each of the horizontal lines meets the braid $m$ times and the regions $m+1$ times, it corresponds to $3m+1$ edges of $\Ocal$.
	\item For the $i$-th horizontal line we orient the corresponding edges and denote them by $x^i_1,\cdots,x^i_{3m+1}$ as in Figure \ref{fig:ocal}, and let $\mathbf{x}^i=(x^i_1,\cdots,x^i_{3m+1})$.
	\end{enumerate}
	\begin{figure}[!h]
	\centering          
	\scalebox{1}{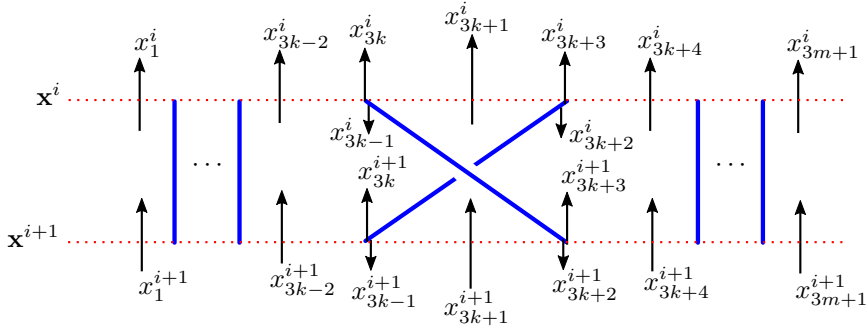}
	\caption{Edges of $\Ocal$ around the $i$-th level of a braid}
	\label{fig:ocal}
	\end{figure}
	Note that there are many overlapped labelings; for instance, in Figure \ref{fig:ocal} we have $x^{i}_j= x^{i+1}_j$ for $j=1,\cdots,3k-2$ and $j=3k+4,\cdots,3m+1$.

We again assign a complex-valued variable to each oriented edge of $\Ocal$ and denote the variable by the same as the edge itself. Hikami and Inoue \cite{hikami2014cluster} related consecutive cluster variables $\mathbf{x}^i$ and $\mathbf{x}^{i+1}$ ($1 \leq i\leq n$) by the equation 
	$$\mathbf{x}^{i+1}=R^{\epsilon_i}_{k_i}(\mathbf{x}^i)$$ where the operator $R^{\pm}_k$ is defined by
	\begin{equation*} 
	R^\pm_k (x_1,\cdots,x_{3m+1}) = \left(x_1,\cdots,x_{3k-3},\ R^\pm(x_{3k-2},\cdots,x_{3k+4}),\ x_{3k+5},\cdots,x_{3m+1}\right).
	\end{equation*} 
	Note that $R^{\pm}_k$ only affects the variables above and below the crossing.
	
	Recall that an initial cluster variable $\mathbf{x}^1 \in \Cbb^{3m+1}$ is called a solution if $\mathbf{x}^1=\mathbf{x}^{n+1}$. Whenever we have a solution $\mathbf{x}^1 \in \Cbb^{3m+1}$, we define the set map $$c_{\mathbf{x}^1} : \Ocal^1 \rightarrow \Cbb$$ by assigning the variable $x^i_j$ to the oriented edge of $\Ocal$ labeled by the same name. The fact that this assignment respects the face identifications in $\Ocal$ follows directly from the definitions of $R^\pm$ and $R^\pm_k$.
	

		\subsection{The obstruction cocycle}	\label{sec:hicocycle}

Let $\Tcal$ be the ideal triangulation of $S^3\setminus (K\cup\{p,q\})$ obtained by decomposing each octahedron of $\Ocal$ into 5 ideal tetrahedra as in Figure \ref{fig:fiveterm}~(left). As explained earlier this induces a triangulation of the boundary of $M=S^3\setminus \nu(K\cup\{p,q\})$.

We define a cocycle $\epsilon_D\in Z^1(\partial M;\pmo)$ on $\partial M$ by assigning signs to the short edges of the truncated tetrahedra. Note that each short edge either lies in the top/bottom of a truncated octahedron, or on one of the sides. We shall call the edges \emph{top/bottom-edges} or \emph{side-edges} accordingly. We assign signs to the top/bottom edges as indicated in Figure~\ref{fig:octahedron} and assign $+1$ do all of the side edges. This is clearly a cocycle, which respects the face pairings and thus gives rise to a cocycle in $\epsilon_D \in Z^1(\partial M;\pmo)$ as desired. We stress that $\epsilon_D$ depends on the decomposition of $M$, in particular the choice of a braid presentation $D$ of $K$.
\begin{figure}[!h]  
	\centering
	\scalebox{1}{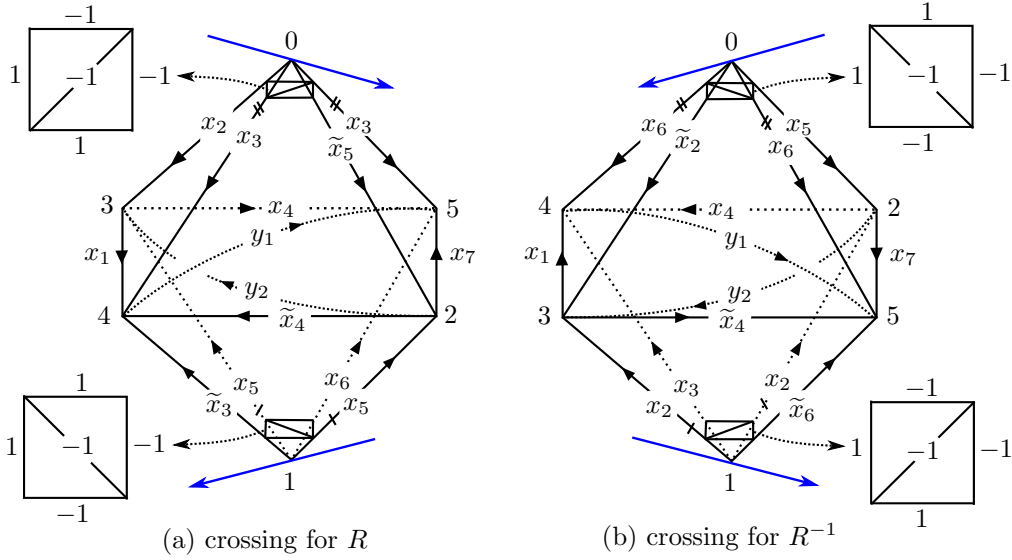}
	\caption{An ideal octahedron at a crossing}
	\label{fig:octahedron}
\end{figure}

The cocycle $\epsilon_D$ is illustrated in Figure~\ref{fig:boundarycocycle}, where $\mu$ and $\lambda_{bf}$ denote a meridian and the black-board framed longitude of $K$, respectively. In particular, $\epsilon_D$ induces the homomorphism $\overline{\epsilon}_D : \pi_1(\nu(K)) \rightarrow \pmo$ that maps $\mu$ to $-1$ and $\lambda_{bf}$ to $1$.
\begin{figure}[!h]
	\centering          
	\scalebox{1}{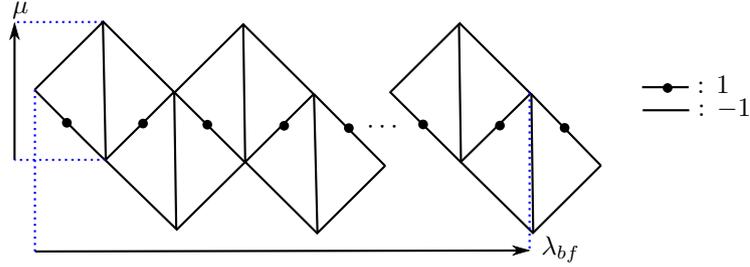}
	\caption{Configuration of $\epsilon_D$ on the boundary torus}
	\label{fig:boundarycocycle}
\end{figure}

	\subsection{Proof of Theorem \ref{thm:int1}}		 \label{sec:proof}
	Let us consider an octahedron of $\Ocal$. We index the vertices by $\{0,\cdots,5\}$ and denote the oriented edges as in Figure \ref{fig:octahedron}. We  compute the Ptolemy relation (\ref{eqn:obs}) for each of the ideal tetrahedra with the obstruction cocycle  $\sigma_D:=d(\epsilon_D) \in Z^2(M,\partial M;\pmo)$. Recall that the map $d$ is given in the equation~\eqref{eqn:d}. For example, the tetrahedron with vertices $\{0,3,4,5\}$ in Figure \ref{fig:octahedron}(a) gives
		\allowdisplaybreaks
		\begin{equation*}
			\begin{array}{rrcl}
			& \sigma_4c(l_{04})c(l_{35})&=&\sigma_5\, c(l_{05})c(l_{34})+ \sigma_3c(l_{03}) c(l_{45}) \\ \Leftrightarrow & (-1)x_3 x_4 &=&  x_3 x_1+  (-1)x_2 y_1 \\
		\Leftrightarrow& x_2 y_1 &=&x_3 x_4+ x_1 x_3.
			\end{array}
		\end{equation*}
		Similarly, we obtain the following from the other ideal tetrahedra:  
		\allowdisplaybreaks
		\begin{equation*}
			\begin{array}{rrcl}
			 	\{0,3,4,5\} :& x_2 y_1&=& x_3 x_4+ x_1 x_3\\
				\{1,2,3,5\} :& x_6 y_2&=&x_5 x_7+x_4 x_5 \\
				\{2,3,4,5\} :& x_4 \widetilde{x}_4&=& x_1 x_7+y_1 y_2  \\
				\{0,2,4,5\} :& \widetilde{x}_5 y_1 &=& x_3 \widetilde{x}_4 + x_3x_7\\
				\{1,2,3,4\} :& \widetilde{x}_3 y_2&=&x_5 \widetilde{x}_4+x_1 x_5 
			\end{array}
		\end{equation*} for Figure \ref{fig:octahedron}(a) and 
		\begin{equation*}
			\begin{array}{rrcl}
			 	\{0,2,4,5\} : & y_1 x_5&=&x_4 x_6+x_6 x_7\\
				\{1,2,3,4\} : & x_3 y_2&=&x_1 x_2+x_2 x_4 \\
				\{2,3,4,5\} : & x_4 \widetilde{x}_4&=&y_1 y_2+x_1 x_7  \\
				\{0,3,4,5\} : & \widetilde{x}_2 y_1 &=& x_6 \widetilde{x}_4 + x_1x_6\\
				\{1,2,3,5\} : & \widetilde{x}_6 y_2 &=& x_2 x_7  +x_2 \widetilde{x}_4
			\end{array}
		\end{equation*}	for Figure \ref{fig:octahedron}(b).

	 	Considering $x_1,\cdots,x_7$ as given variables, we obtain
	 	\begin{equation}\label{eqn:ya}
	 		(y_1,y_2)=\left(\dfrac{x_3(x_1 +x_4)}{x_2},\ \dfrac{x_5( x_4+ x_7)}{x_6}\right)
	 	\end{equation}
	 	\begin{equation*} 
	 			 (\widetilde{x}_3,\widetilde{x}_4,\widetilde{x}_5)=\begin{pmatrix}
	 			 \dfrac{x_1 x_3 x_5+x_3 x_4 x_5+x_1 x_2 x_6}{x_2 x_4}\\[1em]
	 			 \dfrac{x_1 x_3 x_4 x_5+ x_3 x_4^2 x_5+x_1 x_3 x_5 x_7 +x_3 x_4 x_5 x_7+ x_1 x_2 x_6 x_7}{x_2 x_4 x_6}\\[1em]
	 			 \dfrac{x_3 x_4 x_5+x_3 x_5 x_7+x_2 x_6 x_7}{x_4 x_6}
	 			 \end{pmatrix}^T
	 	\end{equation*} for Figure \ref{fig:octahedron}(a) and 
	 	\begin{equation}\label{eqn:yb}
	 		(y_1,y_2)=\left(\dfrac{x_6(x_4 +x_7)}{x_5},\ \dfrac{x_2( x_1+ x_4)}{x_3}\right)	 
	 	\end{equation}
	 	\begin{equation*} 
			 	(\widetilde{x}_2,\widetilde{x}_4,\widetilde{x}_6)=\begin{pmatrix}
			 	\dfrac{x_1 x_3 x_5+x_1 x_2 x_6+x_2 x_4 x_6}{x_3 x_4}\\[1em]
			 	\dfrac{x_1 x_2 x_4 x_6+ x_2 x_4^2 x_6+x_1 x_3 x_5 x_7 +x_1 x_2 x_6 x_7+ x_2 x_4 x_6 x_7}{x_3 x_4 x_5}\\[1em]
			 	\dfrac{x_2 x_4 x_6+x_3 x_5 x_7+x_2 x_6 x_7}{x_4 x_5}
			 	\end{pmatrix}^T
	 	\end{equation*}
	 	for Figure \ref{fig:octahedron}(b). 
	 	Letting $\widetilde{x}_1:=x_1,\ \widetilde{x}_2:=x_5,\ \widetilde{x}_6:=x_3,\ \widetilde{x}_7:=x_7$ for Figure \ref{fig:octahedron}(a) and $\widetilde{x}_1:=x_1,\ \widetilde{x}_3:=x_6,\ \widetilde{x}_5:=x_2,\ \widetilde{x}_7:=x_7$ for Figures \ref{fig:octahedron}(b) (note that these equations come from Figures \ref{fig:block} and \ref{fig:tempt}), we obtain  
	 	 		\begin{equation} \label{eqn:rp}
	 	 		\begin{pmatrix}
	 	 		\widetilde{x}_1\\[0.5em]
	 	 		\widetilde{x}_2\\[0.5em]
	 	 		\widetilde{x}_3\\[0.5em]
	 	 		\widetilde{x}_4\\[0.5em]
	 	 		\widetilde{x}_5\\[0.5em]
	 	 		\widetilde{x}_6\\[0.5em]
	 	 		\widetilde{x}_7\\[0.5em]
	 	 		\end{pmatrix}^T \mkern-10mu =
	 	 		\begin{pmatrix}
	 			 x_1\\[0.5em]
	 			 x_5\\[0.5em]
	 			 \dfrac{x_1 x_3 x_5+x_3 x_4 x_5+x_1 x_2 x_6}{x_2 x_4}\\[1em]
	 			 \dfrac{x_1 x_3 x_4 x_5+ x_3 x_4^2 x_5+x_1 x_3 x_5 x_7 +x_3 x_4 x_5 x_7+ x_1 x_2 x_6 x_7}{x_2 x_4 x_6}\\[1em]
	 			 \dfrac{x_3 x_4 x_5+x_3 x_5 x_7+x_2 x_6 x_7}{x_4 x_6}\\[1em]
	 			 x_3 \\[0.5em]
	 			 x_7 \\
	 			 \end{pmatrix}^T \mkern-10mu =
	 			 R
	 			 \begin{pmatrix}
	 			 x_1\\[0.5em]
	 			 x_2\\[0.5em]
	 			 x_3\\[0.5em]
	 			 x_4\\[0.5em]
	 			 x_5\\[0.5em]
	 			 x_6\\[0.5em]
	 			 x_7\\[0.5em]
	 			 \end{pmatrix}^T 
	 			 \end{equation}
	 			 for Figure \ref{fig:octahedron}(a) and
	 			 \begin{equation} \label{eqn:rm}
	 			 \begin{pmatrix}
	 			 \widetilde{x}_1\\[0.5em]
	 			 \widetilde{x}_2\\[0.5em]
	 			 \widetilde{x}_3\\[0.5em]
	 			 \widetilde{x}_4\\[0.5em]
	 			 \widetilde{x}_5\\[0.5em]
	 			 \widetilde{x}_6\\[0.5em]
	 			 \widetilde{x}_7\\[0.5em]
	 			 \end{pmatrix}^T \mkern-10mu 
	 			 =\begin{pmatrix}
	 			 x_1\\[0.5em]
	 			 \dfrac{x_1 x_3 x_5+x_1 x_2 x_6+x_2 x_4 x_6}{x_3 x_4}\\[1em]
	 			 x_6\\[0.5em]		
	 			 \dfrac{x_1 x_2 x_4 x_6+ x_2 x_4^2 x_6+x_1 x_3 x_5 x_7 +x_1 x_2 x_6 x_7+ x_2 x_4 x_6 x_7}{x_3 x_4 x_5}\\[1em]
	 			 x_2\\[0.5em]
	 			 \dfrac{x_3 x_5 x_7+x_2 x_4 x_6+x_2 x_6 x_7}{x_4 x_5}\\[1em]
	 			 x_7
	 			 \end{pmatrix}^T \mkern-10mu =
	 			 R^{-1} \mkern-5mu
	 			 \begin{pmatrix}
	 			 x_1\\[0.5em]
	 			 x_2\\[0.5em]
	 			 x_3\\[0.5em]
	 			 x_4\\[0.5em]
	 			 x_5\\[0.5em]
	 			 x_6\\[0.5em]
	 			 x_7\\[0.5em]
	 			 \end{pmatrix}^T 
	 			 \end{equation}
	 		for Figure \ref{fig:octahedron}(b). The equations (\ref{eqn:rp}) and (\ref{eqn:rm}) exactly coincide with the definition of $R^\pm$ in \cite{hikami2015braids}. See \cite[Equation (2.12)]{hikami2015braids}.
	 		
	 		\begin{rmk}
	 			Subdividing the octahedron into five tetrahedra as in Figure~\ref{fig:octahedron} corresponds to taking the form of the $R$-operator given in the equation (2.14) instead of (2.9) of \cite{hikami2015braids}.
	 		\end{rmk}
	 	
		

		Now let $D$ be a braid of length $n$ and width $m$. Let  $c_{\mathbf{x}^1} : \Ocal^1 \rightarrow \Cbb$ be the set map induced from a solution $\mathbf{x}^1 \in  \Cbb^{3m+1}$ as in Section \ref{sec:octa}. Recall that $\Tcal$ has two additional edges per crossing compared to $\Ocal$. We extend the set map to $c_{\mathbf{x}^1} : \Tcal^1 \rightarrow \Cbb$ by defining the values on the added edges using the equations (\ref{eqn:ya}) and (\ref{eqn:yb}). 
		We say that a solution $\mathbf{x}^1$ is \emph{non-degenerate} if \[c_{\mathbf{x}^1}(e) \neq 0 \] for all $e\in \Tcal^1$. One can easily check from  the equations (\ref{eqn:ya}) and (\ref{eqn:yb}) that this is equivalent to the following.
		
		\begin{defn}\label{defn:nond}
			A solution $\mathbf{x}^1$ is said to be \emph{non-degenerate} if every cluster variable $\mathbf{x}^i=(x^i_1,\cdots,x^i_{3m+1})$ satisfies $x^i_j \neq 0$ for all $1 \leq j \leq 3m+1$ and $x^i_{3j-2} \neq - x^i_{3j+1}$ for all $1 \leq j \leq m$.			
		\end{defn}
			
		The previous computation in this section tells us that the set map $c_{\mathbf{x}^1} : \Tcal^1 \rightarrow \Cbb \setminus\{0\}$ induced from a non-degenerated solution $\mathbf{x}^1$ is a point of the Ptolemy variety $P^{\sigma_D}(\Tcal)$ with the obstruction cocycle $\sigma_D \in Z^ 2(M,\partial M;\pmo)$. We have thus proven:
		\begin{prop} A non-degenerate solution $\mathbf{x}^1$ induces a boundary parabolic representation \[\rho_{\mathbf{x}^1} : \pi_1(S^3 \setminus K)=\pi_1(M)\rightarrow \psl\] (up to conjugation) whose obstruction class is $[\sigma_D] \in H^2(M,\partial M;\pmo)$.
		\end{prop} 
		\begin{rmk} A Ptolemy assignment (with obstruction class) determines the shape (or cross-ratios parameters) of the ideal tetrahedra so that they fulfill Thurston's gluing equations. We refer to \cite[\S~12]{garoufalidis2015complex} for details. It implies that the proof of the second part of Theorem 3.1 in \cite{hikami2015braids} is unnecessary.
			
		\end{rmk}

%
	
		\begin{prop}Let $D$ be a braid of length $n$ representing  a knot. Then $[\sigma_D]$ is $(-1)^n$ under the isomorphism $H^2(M,\partial M;\pmo) \simeq \pmo$.
		\end{prop}
		\begin{proof} It suffices to show that $\overline{\epsilon}_D(\lambda)=(-1)^n$, where $\lambda$ is the canonical longitude and  $\overline{\epsilon}_D$ denotes the homomorphism induced from $\epsilon_D \in Z^1(\partial M ;\pmo)$. Recall Section \ref{sec:hicocycle} that we have $\overline{\epsilon}_D(\mu)=-1$ and $\overline{\epsilon}_D(\lambda_{bf})=1$ for the meridian $\mu$ and blackboard framed longitude $\lambda_{bf}$. We thus obtain 
		\begin{talign*}
			\overline{\epsilon}_D(\lambda) &= \overline{\epsilon}_D(\lambda_{bf})\  \overline{\epsilon}_D(\mu)^{-w(D)}\\
			&=\overline{\epsilon}_D(\lambda_{bf})\  \overline{\epsilon}_D(\mu)^{-n} = (-1)^n.
		\end{talign*}
		Here $w(D)$ denotes the writhe of the closure of $D$ which is congruent to the length $n$ in modulo $2$.
		\end{proof}

\section{The existence of a non-degenerate solution} \label{sec:dec}

Let $\widetilde{M}$ be the universal cover of $M=S^3 \setminus \nu(K \cup\{p,q\})$ and $\widehat{\widetilde{M}}$ be the space obtained from $\widetilde{M}$ by collapsing each boundary component to a point. We denote by $I(\widetilde{M})$ the 	set of these points. Note 
that $\pi_1(M)$ acts on $I(\widetilde{M})$.

\begin{defn} For a $(\psl,P)$-representation  $\rho:\pi_1(M)\rightarrow \psl$,  a \emph{decoration} $\Dcal : I(\widetilde{M}) \rightarrow \psl/P$ is a $\rho$-equivalent assignment, i.e., $\Dcal(\gamma \cdot v) = \rho(\gamma) \Dcal(v)$ for all $\gamma \in \pi_1(M)$ and $v \in I(\widetilde{M})$. 
\end{defn}

Recall that  $\psl/P$ denotes the (left) $P$-coset space where $P$ is the subgroup of $\psl$ consisting of upper triangular matrices with ones on the diagonal.
We may identify a $P$-coset $gP$ with a vector $g\binom{1}{0}$ which is well-defined up to sign. In particular, by $\textrm{det}(gP,hP)$ we mean $\textrm{det}\left(g\binom{1}{0}, h\binom{1}{0}\right) \in \Cbb /\pmo$.

We now fix a braid presentation $D$ of a knot $K$ and let $\Tcal$ be the ideal triangulation of $S^3 \setminus (K\cup\{p,q\})$ given as in Section \ref{sec:main}.
For any decoration $\Dcal$ we define an assignment $c : \Tcal^1 \rightarrow \Cbb/{\pmo}$ by \[c(e) = \textrm{det} \left(\Dcal(v_1),\Dcal(v_2)\right)\] for $e \in \Tcal^1$ where $v_1$ and $v_2 \in I(\widetilde{M})$ are endpoints of a lift of $e$. Note that $c(e)$ does not depend on the choice of a lift of $e$, since $\Dcal$ is $\rho$-equivariant. 

\begin{prop}\label{thm:dec} For a non-trivial $(\psl,P)$-representation $\rho:\pi_1(M)\rightarrow \psl$, there exists a decoration $\Dcal$ such that the induced assignment $c$ satisfies $c(e) \neq 0 $ for all $e \in \Tcal^1$.	
\end{prop}
The proof of Proposition~\ref{thm:dec} (see Section~\ref{sec:deco} for details) relies on the following basic facts: (i) every edge of $\Tcal$ are connected to either $p$ or $q$; (ii) a decoration on the lifts of $p$ and $q$ can be chosen freely and independently (respecting $\rho$-equivalence only). The observation that (i) and (ii) implies Proposition~\ref{thm:dec} was first pointed out to the authors by Seonhwa Kim. 
We also note that there are edges connecting $p$ (or $q$) to itself and this is the reason why we can not detect the trivial representation.
Namely, these edges become generators in the Wirtinger presentation (see Figure \ref{fig:developing2}~(left)) and thus the
image of the generators under $\rho$ must be non-trivial.

\begin{rmk} \label{rmk:essen} The ideal triangulation used in \cite{hikami2015braids} does not satisfy fact (i) above. The edge $x_c$ in Figure 4 of \cite{hikami2015braids} joins an ideal vertex corresponding to the knot to itself. In this case, Proposition \ref{thm:dec} may not hold, whenever the closure of $D$ has a kink. Therefore, it is essential to use another subdivision of an octahedron (for instance, as in Figure \ref{fig:octahedron}) so that fact  (i) holds. 
\end{rmk}
Proposition \ref{thm:dec} implies the existence of a non-degenerate solution desired as in Theorem \ref{thm:int2}. More precisely, the following holds.

\begin{thm} \label{thm:exist} Let $\sigma_D \in Z^2(M,\partial M;\pmo)$ be the cocycle given as in Section \ref{sec:main}. If a non-trivial $(\psl,P)$-representation $\rho$ has the   obstruction class $[\sigma_D] \in H^2(M,\partial M;\pmo)$, then there exists a point $c\in P^{\sigma_D}(\Tcal)$ such that $\rho_c$ coincides with $\rho$ up to conjugation.
\end{thm}

\begin{proof} Let $\Dcal$ be a decoration as in Proposition \ref{thm:dec}. Whenever one chooses a sign of each $c(e)$, it is known that $c :\Tcal^1\rightarrow \Cbb \setminus \{0\}$ is a point of $P^\sigma(\Tcal)$ for some $\sigma \in Z^2(M,\partial M;\pmo)$ such that $\rho_c=\rho$ up to conjugation. In particular, the obstruction class of $\rho$ is $[\sigma] \in H^2(M,\partial M;\pmo)$. Then the theorem follows from the fact that if $\sigma_0$ and $\sigma_1 \in Z^2(M,\partial M;\pmo)$ satisfy $[\sigma_0]=[\sigma_1]$, then two varieties $P^{\sigma_0}(\Tcal)$ and $P^{\sigma_1}(\Tcal)$ are canonically isomorphic.	
\end{proof}

As we computed in Section \ref{sec:proof}, the class $[\sigma_D]$ viewed as an element of $\pmo$ coincides with $(-1)^n$ where $n$ is the length of $D$. We therefore obtain Theorem \ref{thm:int2} as a consequence of Theorem \ref{thm:exist}.

\subsection{Proof of Proposition \ref{thm:dec}} \label{sec:deco}
 We first consider edges, say $e_1,\cdots,e_m$, of $\Tcal$ that join $p$ and $q$. We orient these edges from $q$ to $p$. We choose a lift $\widetilde{e}_j$ of each $e_j$ so that their terminal points agree as in Figure \ref{fig:developing}. We denote by $\widetilde{p}$ the terminal point  and by $\widetilde{q}_j$ the initial point of $\widetilde{e}_j$. From $\rho$-equivariance of $\Dcal$, we have $$\Dcal(\widetilde{q}_{j})=\rho(g) \Dcal(\widetilde{q}_k) $$
for some $g\in \pi_1(M)$.  From elementary covering theory one can check that if $e_j \cup e_k$ wraps an arc of $K$ as in  Figure \ref{fig:developing}, then the loop $g$ should be the Wirtinger generator corresponding to the arc. Note that $c(e_k) \neq 0 $ if and only if $\textrm{det}(\Dcal(\widetilde{p}),\Dcal(\widetilde{q}_k)) \neq 0$.
\begin{figure}[!h]
	\centering
	\scalebox{1}{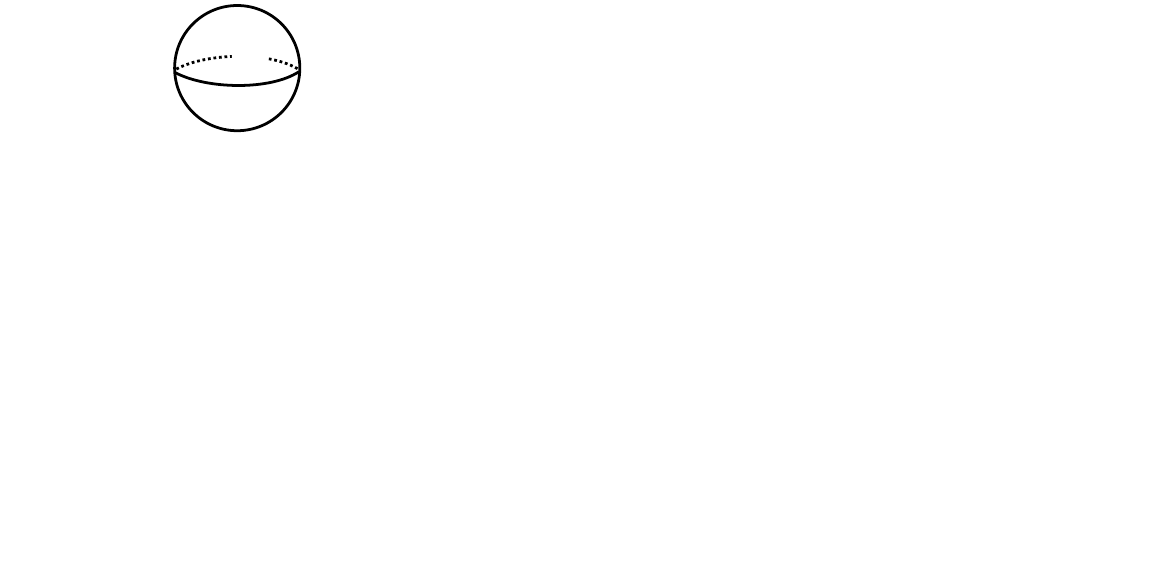} 
	\caption{Local configuration of a lift.}
	\label{fig:developing}
\end{figure}

We then consider edges of $\Tcal$ that are connected to the knot $K$; for example, edges $x$ and $y$ as in Figure \ref{fig:developing}. We consider an ideal triangle in $S^3 \setminus (K \cup \{p,q\})$ with edges $x,y,e_k$ as in Figure \ref{fig:developing}, and its lift so that $p$ corresponds to the point $\widetilde{p}$. We denote the edges of the lift by $\widetilde{x},\widetilde{y},\widetilde{e}_k$. Since the terminal point, $\widetilde{r}$, of $\widetilde{x}$ (or $\widetilde{y}$) is fixed by the Wirtinger generator $g$, we obtain
\[\Dcal(\widetilde{r})=\Dcal(g \cdot \widetilde{r}) =\rho(g) \Dcal(\widetilde{r}).\]
Since $\textrm{tr}(\rho(g))=\pm2$ and $\rho(g) \neq \mathrm{Id}$, otherwise $\rho$ should be a trivial representation, $\rho(g)$ has a unique eigenvector up to scaling. It thus follows that $c(x)=\textrm{det}(\Dcal(\widetilde{p}),\Dcal(\widetilde{r})) \neq 0$ if and only if $\Dcal(\widetilde{p})$ is not an eigenvector of $\rho(g)$. Similarly, $c(y) \neq 0 $ if and only if $\Dcal(\widetilde{q}_k)$ is not an eigenvector of $\rho(g)$.

We finally consider edges of $\Tcal$  joining $q$ (or $p$) to itself; for example, an edge $x$ as in Figure \ref{fig:developing2}.
We consider an ideal triangle in $S^3 \setminus (K \cup \{p,q\})$ with edges $e_j,e_k,x$ as in Figure \ref{fig:developing2}, and its lift so that $p$ corresponds to the point $\widetilde{p}$. We  denote the edges of the lift by $\widetilde{e}_j,\widetilde{e}_k,\widetilde{x}$.
 It directly follows that $c(x) \neq 0$ if and only if \[\textrm{det}(\Dcal(\widetilde{q}_j),\Dcal(\widetilde{q}_k))=\textrm{det}(\rho(g)\Dcal(\widetilde{q}_k),\Dcal(\widetilde{q}_k)) \neq0.\] Again, this is equivalent to the condition that $\Dcal(\widetilde{q}_k)$ is not an eigenvector of $\rho(g)$. 
	\begin{figure}[!h]
		\centering
		\scalebox{1}{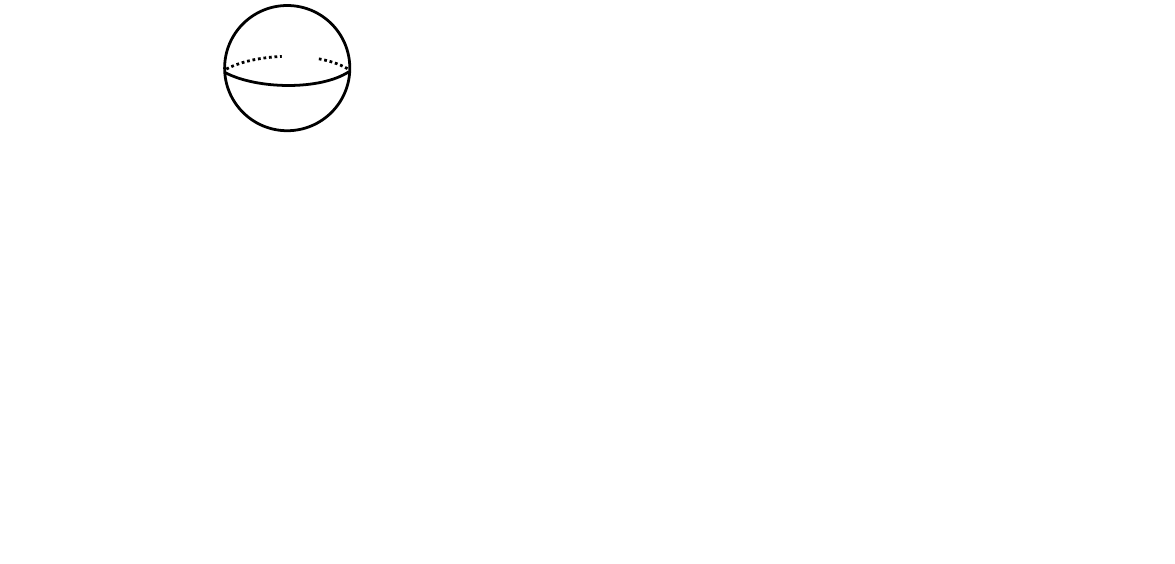} 
		\caption{Local configuration of a lift.}
		\label{fig:developing2}
	\end{figure}

Let us sum up the conditions. To be precise, we enumerate the Wirtinger generators by $g_1,\cdots,g_{l}$. Our desired decoration as in Proposition~\ref{thm:dec} should satisfy (i) $ \textrm{det}(\Dcal(\widetilde{p}), \Dcal(\widetilde{q}_j)) \neq 0$; (ii) $\Dcal(\widetilde{p})$ is not an eigenvector of $\rho(g_i)$; (iii) $\Dcal(\widetilde{q}_j)$ is not an eigenvector of $\rho(g_i)$ 
for all $1 \leq j \leq m$ and $1 \leq i \leq l$.
Since we can choose $\Dcal(\widetilde{p})$ and one of $\Dcal(\widetilde{q}_j)$'s freely, such a decoration exists.

\subsection{Explicit computation from a representation} \label{sec:comp}
Let $D$ be a braid presentation of a knot $K$ and let $\rho :\pi_1(S^3 \setminus K)\rightarrow \psl$ be a non-trivial $(\psl,P)$-representation whose obstruction class is $(-1)^n$, where $n$ is the length of $D$. We devote this subsection to present an explicit formula for computing a solution.

Let $\widetilde{\rho}$ be an $\sl$-lift  of $\rho$ satisfying
\begin{equation*}
	\widetilde{\rho}(\mu)=\begin{pmatrix}
	-1 & * \\ 0 & -1
	\end{pmatrix} \neq -\textrm{Id} \quad \textrm{and} \quad
	\widetilde{\rho}(\lambda)=\begin{pmatrix}
	(-1)^n & * \\ 0 & (-1)^n
	\end{pmatrix}
\end{equation*}  (recall Proposition \ref{prop:knot}).
%
%
 We index the regions of the closure of $D$ by $1 \leq j \leq n+2$ and the arcs by $1 \leq i \leq n$.
We then assign a non-zero column vector $V_j$ to the $j$-th region so that these vectors satisfy \begin{equation}\label{eqn:region} V_{j_2} =\widetilde{\rho}(g_i)^{-1} V_{j_1} \end{equation} for Figure \ref{fig:rule}~(left) where $m_i$ is the Wirtinger generator corresponding to the $i$-th arc. The region-colorings are well-determined whenever an initial vector is chosen arbitrarily.
Remark that $V_j$ corresponds to $\Dcal(\widetilde{q}_j)$ in Section \ref{sec:deco}.

 We also assign a non-zero column vector $H_i$ to the $i$-th arc so that these vectors satisfy $\widetilde{\rho}(g_i)H_i = -H_i$ for $1 \leq i \leq m$ (recall that the eigenvalue of $\widetilde{\rho}(g_i)$ is $-1$) and 
\begin{equation}\label{eqn:arcclr}
H_{i_3} = \widetilde{\rho}(g_{i_2})^{-1}H_{i_1}
\end{equation} for Figure \ref{fig:rule}~(right). 
We remark that the fact that the eigenvalue of $\widetilde{\rho}(\lambda_{bf})$ is $1$ (equivalently, the eigenvalue of $\widetilde{\rho}(\lambda)$ is $(-1)^n$) is required here.

\begin{figure}[!h]
	\centering
	\scalebox{1}{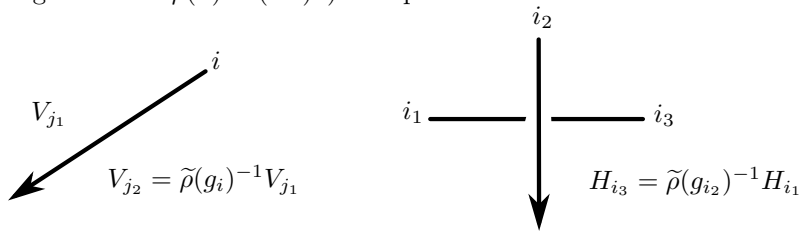}
	\caption{Rules for region- and arc-colorings.}
	\label{fig:rule}
\end{figure}

Recall that the octahedral decomposition $\Ocal$ has $3n+2$ edges; (i) $n$ of them, called \emph{over-edges}, stand above the knot; (ii) other $n$ of them, called \emph{under-edges}, stand below the knot; (iii) last $n+2$ of them, called \emph{regional edges}, stand on the regions. See Figure~\ref{fig:rule2}. 
We choose an additional non-zero column vector $W$ (which corresponds to $\Dcal(\widetilde{p})$ in Section \ref{sec:deco}) and define the set map $c : \Ocal^1 \rightarrow \Cbb$ as follows.
\begin{itemize}
	\item [(i)] $c(e)= \textrm{det}(H_i,W)$ if $e$ is the over-edge standing over the $i$-th arc;
	\item [(ii)] $c(e)= \textrm{det}(V_j,H_i)$ if $e$ is the under-edge standing below the $i$-th arc whose left-side region is indexed by $j$;
	\item  [(iii)] $c(e) = \textrm{det}(V_j,W)$ if $e$ is the regional edge corresponding the $j$-th region.
\end{itemize} Here we oriented the edge $e$ as in Figure \ref{fig:rule2}.

\begin{figure}[!h]
	\centering
	\scalebox{1}{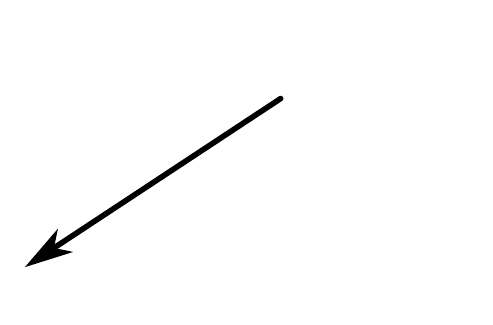}
	\caption{Edges of $\Ocal$ with $c$-values.}
	\label{fig:rule2}
\end{figure} 

We again extend the above set map to $c : \Tcal^1 \rightarrow \Cbb$ by using the equations (\ref{eqn:ya}) and (\ref{eqn:yb}). 
As we showed in Section \ref{sec:deco} for a generic choice of $W$ and $V_j$'s, we have $c(e) \neq 0 $ for all $e\in \Tcal^1$.

%


\begin{exam}[The $4_1$ knot with a kink] Let us consider a braid of the knot $4_1$ as in Figure \ref{fig:braid}.
	\begin{figure}[!h]
		\centering
		\scalebox{1}{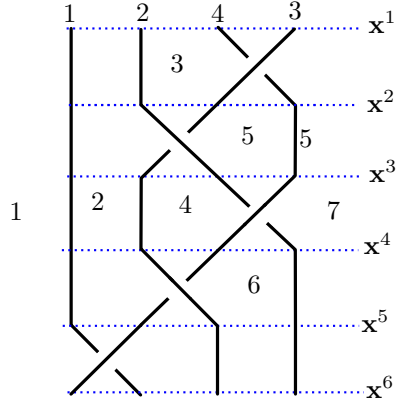}
		\caption{A braid presentation of the $4_1$ knot.}
		\label{fig:braid}
	\end{figure} The geometric representation $\rho$ lifts to an $\sl$-representation $\widetilde{\rho}$ such that
	\begin{equation*}
		\begin{array}{l}
		\widetilde{\rho}(g_1)=\widetilde{\rho}(g_2)=\begin{pmatrix}
		-1 & -1 \\ 0 & -1
		\end{pmatrix},\quad \widetilde{\rho}(g_3)=\begin{pmatrix}
		-1 & 0 \\ -\lambda & -1
		\end{pmatrix}\\[10pt]
		 \widetilde{\rho}(g_4)=\begin{pmatrix}
		-1-\lambda & \lambda \\ -\lambda & -1+\lambda
		\end{pmatrix},\quad \widetilde{\rho}(m_4)=\begin{pmatrix}
		-2& \lambda \\ 	-1+\lambda & 0
		\end{pmatrix}
		\end{array}
	\end{equation*} where $\lambda^2-\lambda+1=0$. 
	
	We enumerate the arcs and regions of the closure of the braid as in Figure \ref{fig:braid}.
	Choosing the vector $H_1=\binom{1}{0}$, the equation \eqref{eqn:arcclr} gives
	\begin{equation*}
	\begin{array}{ll}
	H_2=\widetilde{\rho}(g_2)^{-1} H_1=\binom{-1}{0},& H_3=\widetilde{\rho}(g_5)^{-1}H_2=\binom{0}{-1+\lambda}\\[10pt]
	H_4=\widetilde{\rho}(g_2) H_3=\binom{1-\lambda}{1-\lambda},& H_5=\widetilde{\rho}(g_3)^{-1}H_4=\binom{-1+\lambda}{\lambda}.
	\end{array}
	\end{equation*}
	Similarly, letting the vector $V_1=\binom{\alpha}{\beta}$ for some $\alpha,\beta \in \Cbb$, the equation \eqref{eqn:region} gives
	\begin{equation*}
	\begin{array}{ll}
	V_2=\widetilde{\rho}(g_1)^{-1} V_1=\binom{-\alpha+\beta}{-\beta},&
	V_3=\widetilde{\rho}(g_2)^{-1} V_2=\binom{\alpha-2\beta}{\beta}\\[10pt]
	V_4=\widetilde{\rho}(g_4)^{-1} V_2=\binom{\alpha(1-\lambda)+\beta(-1+2\lambda)}{-\alpha \lambda+\beta(1+2\lambda)},& V_5=\widetilde{\rho}(g_3)^{-1}V_3=\binom{-\alpha+2\beta}{\alpha\lambda-\beta(1+2\lambda)}\\[10pt]
	V_6=\widetilde{\rho}(g_5)^{-1}V_4=\binom{\alpha(-1+\lambda)+\beta(2-3\lambda)}{\alpha\lambda-\beta(1+3\lambda)}, &
	V_7=\widetilde{\rho}(g_5)^{-1}V_5=\binom{\alpha(1-\lambda)+\beta(-2+3\lambda)}{-\alpha(1+\lambda)+2\beta(2+\lambda)}.
	\end{array}
	\end{equation*} Then finally, letting the vector $W=\binom{\gamma}{1}$ for some $\gamma \in \Cbb$, we obtain the cluster variables $\mathbf{x}^1,\cdots,\mathbf{x}^5$ as follows. We note that a generic choice for $\alpha, \beta$, and $\gamma$ gives a non-degenerate solution. Here we abbreviate $\textrm{det}(\cdot,\cdot)$ by $|\cdot,\cdot|$.
	\begin{equation*}
		\mathbf{x}^1
		=\begin{pmatrix}
		|V_1, W|\\|V_2, H_1|\\|H_1, W|\\|V_2, W|\\|V_3, H_2|\\|H_2, W|\\|V_3, W|\\|V_6, H_4|\\ |H_4, W|\\|V_6, W|\\|V_7, H_3|\\|H_3, W|\\|V_7, W|
		\end{pmatrix}^T=\begin{pmatrix}
			\alpha -\beta  \gamma \\
			\beta \\
			1\\
			-\alpha +\beta  \gamma +\beta\\
			\beta \\
			-1 \\
			\alpha -\beta  (\gamma +2)\\
			(\lambda -1) (\alpha -3 \beta ) \\
			(\gamma -1) (\lambda -1) \\
			\alpha  (-\gamma  \lambda +\lambda -1)+\beta  (3 (\gamma -1) \lambda +\gamma +2) \\
			\alpha  \lambda -\beta  (2 \lambda +1)\\
			\gamma -\gamma  \lambda \\
			\alpha  ((\gamma -1) \lambda +\gamma +1)-\beta  (2 \gamma  (\lambda +2)-3 \lambda +2)
		\end{pmatrix}^T
	\end{equation*}
	\begin{equation*}
		\mathbf{x}^2
		=\begin{pmatrix}
		|V_1, W|\\|V_2, H_1|\\|H_1, W|\\|V_2, W|\\|V_3, H_2|\\|H_2, W|\\|V_3, W|\\|V_5, H_3|\\ |H_3, W|\\|V_5, W|\\|V_7, H_5|\\|H_5, W|\\|V_7, W|
		\end{pmatrix}^T=\begin{pmatrix}
			\alpha -\beta  \gamma \\
			\beta \\
			1 \\
			-\alpha +\beta  \gamma +\beta \\
			\beta \\
			-1\\
			\alpha -\beta  (\gamma +2)\\
			\lambda ^2 (-(\alpha -2 \beta ))\\
			\gamma -\gamma  \lambda \\
			\beta  (2 \gamma  \lambda +\gamma +2)-\alpha  (\gamma  \lambda +1)\\
			(\lambda -1) (\alpha -3 \beta )\\
			-\gamma  \lambda +\lambda -1\\
			\alpha  ((\gamma -1) \lambda +\gamma +1)-\beta  (2 \gamma  (\lambda +2)-3 \lambda +2)	\end{pmatrix}^T
		\end{equation*}
	\begin{equation*}
	\mathbf{x}^3
	=\begin{pmatrix}
	|V_1, W|\\|V_2, H_1|\\|H_1, W|\\|V_2,W|\\|V_4, H_4|\\|H_4, W|\\|V_4, W|\\|V_5, H_2|\\ |H_2, W|\\|V_5, W|\\|V_7, H_5|\\|H_5, W|\\|V_7, W|
	\end{pmatrix}^T=\begin{pmatrix}
	\alpha -\beta  \gamma \\
	\beta \\ 1 \\ 
	-\alpha +\beta  \gamma +\beta \\
	(\lambda -1) (-(\alpha -2 \beta )) \\
	(\gamma -1) (\lambda -1)\\
	(\gamma -1) \lambda  (\alpha -2 \beta )+\alpha -\beta  (\gamma +1)\\
	\alpha  \lambda -\beta  (2 \lambda +1)\\
	-1 \\
	\beta  (2 \gamma  \lambda +\gamma +2)-\alpha  (\gamma  \lambda +1) \\
	(\lambda -1) (\alpha -3 \beta ) \\
	-\gamma  \lambda +\lambda -1 \\
	\alpha  ((\gamma -1) \lambda +\gamma +1)-\beta  (2 \gamma  (\lambda +2)-3 \lambda +2)
\end{pmatrix}^T
	\end{equation*}
	
	\begin{equation*}
	\mathbf{x}^4
	=\begin{pmatrix}
	|V_1, W|\\|V_2, H_1|\\|H_1, W|\\|V_2, W|\\|V_4, H_4|\\|H_4, W|\\|V_4, W|\\|V_6, H_5|\\ |H_5, W|\\|V_6, W|\\|V_7, H_3|\\|H_3, W|\\|V_7, W|
	\end{pmatrix}^T=\begin{pmatrix}
		\alpha -\beta  \gamma \\
		\beta \\
		1\\
		-\alpha +\beta  \gamma +\beta \\
		(\lambda -1) (-(\alpha -2 \beta ))\\
		(\gamma -1) (\lambda -1)\\
		(\gamma -1) \lambda  (\alpha -2 \beta )+\alpha -\beta  (\gamma +1)\\
		-\beta \\
		-\gamma  \lambda +\lambda -1 \\
		\alpha  (-\gamma  \lambda +\lambda -1)+\beta  (3 (\gamma -1) \lambda +\gamma +2) \\
		\alpha  \lambda -\beta  (2 \lambda +1) \\
		\gamma -\gamma  \lambda \\
		\alpha  ((\gamma -1) \lambda +\gamma +1)-\beta  (2 \gamma  (\lambda +2)-3 \lambda +2)	\end{pmatrix}^T
	\end{equation*}
	
	\begin{equation*}
	\mathbf{x}^5
	=\begin{pmatrix}
	|V_1, W|\\|V_2, H_1|\\|H_1, W|\\|V_2, W|\\|V_3, H_1|\\|H_1, W|\\|V_3, W|\\|V_6, H_4|\\ |H_4, W|\\|V_6, W|\\|V_7, H_3|\\|H_3, W|\\|V_7, W|
	\end{pmatrix}^T=\begin{pmatrix}
		\alpha -\beta  \gamma \\
		\beta \\
		1 \\
		-\alpha +\beta  \gamma +\beta \\
		-\beta \\
		1 \\
		\alpha -\beta  (\gamma +2) \\
		(\lambda -1) (\alpha -3 \beta ) \\
		(\gamma -1) (\lambda -1) \\
		\alpha  (-\gamma  \lambda +\lambda -1)+\beta  (3 (\gamma -1) \lambda +\gamma +2) \\
		\alpha  \lambda -\beta  (2 \lambda +1)\\
		\gamma -\gamma  \lambda \\
		\alpha  ((\gamma -1) \lambda +\gamma +1)-\beta  (2 \gamma  (\lambda +2)-3 \lambda +2)
\end{pmatrix}^T
	\end{equation*}
 
\end{exam}
\bibliographystyle{abbrv}
\bibliography{biblog}
\end{document}

%% file: truncation.pdf_tex
\begingroup%
  \makeatletter%
  \providecommand\color[2][]{%
    \errmessage{(Inkscape) Color is used for the text in Inkscape, but the package 'color.sty' is not loaded}%
    \renewcommand\color[2][]{}%
  }%
  \providecommand\transparent[1]{%
    \errmessage{(Inkscape) Transparency is used (non-zero) for the text in Inkscape, but the package 'transparent.sty' is not loaded}%
    \renewcommand\transparent[1]{}%
  }%
  \providecommand\rotatebox[2]{#2}%
  \ifx\svgwidth\undefined%
    \setlength{\unitlength}{254.15553055bp}%
    \ifx\svgscale\undefined%
      \relax%
    \else%
      \setlength{\unitlength}{\unitlength * \real{\svgscale}}%
    \fi%
  \else%
    \setlength{\unitlength}{\svgwidth}%
  \fi%
  \global\let\svgwidth\undefined%
  \global\let\svgscale\undefined%
  \makeatother%
  \begin{picture}(1,0.39252279)%
    \put(0.18717813,-0.00691984){\color[rgb]{0,0,0}\makebox(0,0)[lb]{\smash{$l_{01}$}}}%
    \put(0.06622296,0.22387133){\color[rgb]{0,0,0}\makebox(0,0)[lb]{\smash{$l_{03}$}}}%
    \put(0.27606259,0.12754332){\color[rgb]{0,0,0}\makebox(0,0)[lb]{\smash{$l_{13}$}}}%
    \put(0,0){\includegraphics[width=\unitlength,page=1]{truncation.pdf}}%
    \put(0.31938822,0.31357114){\color[rgb]{0,0,0}\makebox(0,0)[lb]{\smash{$l_{23}$}}}%
    \put(0.40617483,0.13662901){\color[rgb]{0,0,0}\makebox(0,0)[lb]{\smash{$l_{12}$}}}%
    \put(0.16531076,0.1783955){\color[rgb]{0,0,0}\makebox(0,0)[lb]{\smash{$l_{02}$}}}%
    \put(0,0){\includegraphics[width=\unitlength,page=2]{truncation.pdf}}%
    \put(-0.00832356,0.03106766){\color[rgb]{0,0,0}\makebox(0,0)[lb]{\smash{$0$}}}%
    \put(0.41263378,-0.00042219){\color[rgb]{0,0,0}\makebox(0,0)[lb]{\smash{$1$}}}%
    \put(0.39368661,0.25438531){\color[rgb]{0,0,0}\makebox(0,0)[lb]{\smash{$2$}}}%
    \put(0.2026442,0.37174969){\color[rgb]{0,0,0}\makebox(0,0)[lb]{\smash{$3$}}}%
    \put(0,0){\includegraphics[width=\unitlength,page=3]{truncation.pdf}}%
    \put(0.74996966,0.15776297){\color[rgb]{0,0,0}\makebox(0,0)[lb]{\smash{$\sigma_{2}$}}}%
    \put(0.87835187,0.2217879){\color[rgb]{0,0,0}\makebox(0,0)[lb]{\smash{$\sigma_{0}$}}}%
    \put(0.81709241,0.11267062){\color[rgb]{0,0,0}\makebox(0,0)[lb]{\smash{$\sigma_{3}$}}}%
    \put(0,0){\includegraphics[width=\unitlength,page=4]{truncation.pdf}}%
    \put(0.7344839,0.22099627){\color[rgb]{0,0,0}\makebox(0,0)[lb]{\smash{$\sigma_{1}$}}}%
    \put(0,0){\includegraphics[width=\unitlength,page=5]{truncation.pdf}}%
    \put(0.74458274,0.27532197){\color[rgb]{0,0,0}\makebox(0,0)[lb]{\smash{$s^3_{01}$}}}%
    \put(0,0){\includegraphics[width=\unitlength,page=6]{truncation.pdf}}%
    \put(0.56379929,0.04500686){\color[rgb]{0,0,0}\makebox(0,0)[lb]{\smash{$s^0_{13}$}}}%
    \put(0,0){\includegraphics[width=\unitlength,page=7]{truncation.pdf}}%
    \put(0.83845384,0.05248471){\color[rgb]{0,0,0}\makebox(0,0)[lb]{\smash{$s^1_{30}$}}}%
    \put(0.74445075,-0.01119757){\color[rgb]{0,0,0}\makebox(0,0)[lb]{\smash{$l_{01}$}}}%
    \put(0,0){\includegraphics[width=\unitlength,page=8]{truncation.pdf}}%
  \end{picture}%
\endgroup%

%% file: block.pdf_tex
\begingroup%
  \makeatletter%
  \providecommand\color[2][]{%
    \errmessage{(Inkscape) Color is used for the text in Inkscape, but the package 'color.sty' is not loaded}%
    \renewcommand\color[2][]{}%
  }%
  \providecommand\transparent[1]{%
    \errmessage{(Inkscape) Transparency is used (non-zero) for the text in Inkscape, but the package 'transparent.sty' is not loaded}%
    \renewcommand\transparent[1]{}%
  }%
  \providecommand\rotatebox[2]{#2}%
  \ifx\svgwidth\undefined%
    \setlength{\unitlength}{284.79579415bp}%
    \ifx\svgscale\undefined%
      \relax%
    \else%
      \setlength{\unitlength}{\unitlength * \real{\svgscale}}%
    \fi%
  \else%
    \setlength{\unitlength}{\svgwidth}%
  \fi%
  \global\let\svgwidth\undefined%
  \global\let\svgscale\undefined%
  \makeatother%
  \begin{picture}(1,0.41048349)%
    \put(0,0){\includegraphics[width=\unitlength,page=1]{block.pdf}}%
    \put(0.21157087,0.30092754){\color[rgb]{0,0,0}\makebox(0,0)[lb]{\smash{$Q_1$}}}%
    \put(0.20898556,0.10885532){\color[rgb]{0,0,0}\makebox(0,0)[lb]{\smash{$Q_2$}}}%
    \put(0,0){\includegraphics[width=\unitlength,page=2]{block.pdf}}%
    \put(0.00298258,0.21093914){\color[rgb]{0,0,0}\makebox(0,0)[lb]{\smash{$x_1$}}}%
    \put(0.43582417,0.24320458){\color[rgb]{0,0,0}\makebox(0,0)[lb]{\smash{$x_7$}}}%
    \put(0.3640085,0.26477877){\color[rgb]{0,0,0}\makebox(0,0)[lb]{\smash{$\widetilde{x}_6$}}}%
    \put(0.16338277,0.08103008){\color[rgb]{0,0,0}\makebox(0,0)[lb]{\smash{$\widetilde{x}_2$}}}%
    \put(0,0){\includegraphics[width=\unitlength,page=3]{block.pdf}}%
    \put(0.22711066,0.41870003){\color[rgb]{0,0,0}\makebox(0,0)[lb]{\smash{$x_4$}}}%
    \put(0,0){\includegraphics[width=\unitlength,page=4]{block.pdf}}%
    \put(0.33708453,0.39809093){\color[rgb]{0,0,0}\makebox(0,0)[lb]{\smash{$x_6$}}}%
    \put(0,0){\includegraphics[width=\unitlength,page=5]{block.pdf}}%
    \put(0.0595518,0.39274239){\color[rgb]{0,0,0}\makebox(0,0)[lb]{\smash{$x_3$}}}%
    \put(0.08918867,0.18652026){\color[rgb]{0,0,0}\makebox(0,0)[lb]{\smash{$x_2$}}}%
    \put(0.29939157,0.15282082){\color[rgb]{0,0,0}\makebox(0,0)[lb]{\smash{$x_5$}}}%
    \put(0,0){\includegraphics[width=\unitlength,page=6]{block.pdf}}%
    \put(0.27145701,0.05703228){\color[rgb]{0,0,0}\makebox(0,0)[lb]{\smash{$\widetilde{x}_4$}}}%
    \put(0,0){\includegraphics[width=\unitlength,page=7]{block.pdf}}%
    \put(0.1471202,0.20439937){\color[rgb]{0,0,0}\makebox(0,0)[lb]{\smash{$\widetilde{x}_3$}}}%
    \put(0,0){\includegraphics[width=\unitlength,page=8]{block.pdf}}%
    \put(0.34969776,0.12774049){\color[rgb]{0,0,0}\makebox(0,0)[lb]{\smash{$\widetilde{x}_5$}}}%
  \end{picture}%
\endgroup%

%% file: fiveterm.pdf_tex
\begingroup%
  \makeatletter%
  \providecommand\color[2][]{%
    \errmessage{(Inkscape) Color is used for the text in Inkscape, but the package 'color.sty' is not loaded}%
    \renewcommand\color[2][]{}%
  }%
  \providecommand\transparent[1]{%
    \errmessage{(Inkscape) Transparency is used (non-zero) for the text in Inkscape, but the package 'transparent.sty' is not loaded}%
    \renewcommand\transparent[1]{}%
  }%
  \providecommand\rotatebox[2]{#2}%
  \ifx\svgwidth\undefined%
    \setlength{\unitlength}{317.56080928bp}%
    \ifx\svgscale\undefined%
      \relax%
    \else%
      \setlength{\unitlength}{\unitlength * \real{\svgscale}}%
    \fi%
  \else%
    \setlength{\unitlength}{\svgwidth}%
  \fi%
  \global\let\svgwidth\undefined%
  \global\let\svgscale\undefined%
  \makeatother%
  \begin{picture}(1,0.55449272)%
    \put(0,0){\includegraphics[width=\unitlength,page=1]{fiveterm.pdf}}%
    \put(0.29918259,0.27491228){\color[rgb]{0,0,0}\makebox(0,0)[lb]{\smash{$=$}}}%
    \put(0,0){\includegraphics[width=\unitlength,page=2]{fiveterm.pdf}}%
    \put(0.43003651,0.40474254){\color[rgb]{0,0,0}\makebox(0,0)[lb]{\smash{$x_2$}}}%
    \put(0.63692879,0.39750998){\color[rgb]{0,0,0}\makebox(0,0)[lb]{\smash{$x_3$}}}%
    \put(0.36254275,0.27958535){\color[rgb]{0,0,0}\makebox(0,0)[lb]{\smash{$x_1$}}}%
    \put(0.48675478,0.35962169){\color[rgb]{0,0,0}\makebox(0,0)[lb]{\smash{$\widetilde{x}_6$}}}%
    \put(0.53059635,0.30760225){\color[rgb]{0,0,0}\makebox(0,0)[lb]{\smash{$x_4$}}}%
    \put(0.48256037,0.20574461){\color[rgb]{0,0,0}\makebox(0,0)[lb]{\smash{$x_5$}}}%
    \put(0.57408827,0.20388504){\color[rgb]{0,0,0}\makebox(0,0)[lb]{\smash{$x_6$}}}%
    \put(0.68955951,0.28318518){\color[rgb]{0,0,0}\makebox(0,0)[lb]{\smash{$x_7$}}}%
    \put(0.53199049,0.25241689){\color[rgb]{0,0,0}\makebox(0,0)[lb]{\smash{$\widetilde{x}_4$}}}%
    \put(0.57421589,0.35243403){\color[rgb]{0,0,0}\makebox(0,0)[lb]{\smash{$\widetilde{x}_5$}}}%
    \put(0.42377518,0.15901728){\color[rgb]{0,0,0}\makebox(0,0)[lb]{\smash{$\widetilde{x}_3$}}}%
    \put(0.63174332,0.15078251){\color[rgb]{0,0,0}\makebox(0,0)[lb]{\smash{$\widetilde{x}_2$}}}%
  \end{picture}%
\endgroup%

%% file: tempt.pdf_tex
\begingroup%
  \makeatletter%
  \providecommand\color[2][]{%
    \errmessage{(Inkscape) Color is used for the text in Inkscape, but the package 'color.sty' is not loaded}%
    \renewcommand\color[2][]{}%
  }%
  \providecommand\transparent[1]{%
    \errmessage{(Inkscape) Transparency is used (non-zero) for the text in Inkscape, but the package 'transparent.sty' is not loaded}%
    \renewcommand\transparent[1]{}%
  }%
  \providecommand\rotatebox[2]{#2}%
  \ifx\svgwidth\undefined%
    \setlength{\unitlength}{283.46456693bp}%
    \ifx\svgscale\undefined%
      \relax%
    \else%
      \setlength{\unitlength}{\unitlength * \real{\svgscale}}%
    \fi%
  \else%
    \setlength{\unitlength}{\svgwidth}%
  \fi%
  \global\let\svgwidth\undefined%
  \global\let\svgscale\undefined%
  \makeatother%
  \begin{picture}(1,0.45007076)%
    \put(0,0){\includegraphics[width=\unitlength,page=1]{tempt.pdf}}%
    \put(0.5351141,0.30275192){\color[rgb]{0,0,0}\makebox(0,0)[lb]{\smash{$x_1$}}}%
    \put(0.62550307,0.3642275){\color[rgb]{0,0,0}\makebox(0,0)[lb]{\smash{$x_3$}}}%
    \put(0.62281198,0.29309954){\color[rgb]{0,0,0}\makebox(0,0)[lb]{\smash{$x_2$}}}%
    \put(0.71089484,0.37342159){\color[rgb]{0,0,0}\makebox(0,0)[lb]{\smash{$x_4$}}}%
    \put(0.8555358,0.36629998){\color[rgb]{0,0,0}\makebox(0,0)[lb]{\smash{$x_6$}}}%
    \put(0.8541786,0.29220157){\color[rgb]{0,0,0}\makebox(0,0)[lb]{\smash{$x_5$}}}%
    \put(0.94046292,0.30083886){\color[rgb]{0,0,0}\makebox(0,0)[lb]{\smash{$x_7$}}}%
    \put(0.53530926,0.13373915){\color[rgb]{0,0,0}\makebox(0,0)[lb]{\smash{$\widetilde{x}_1$}}}%
    \put(0.62666859,0.13956386){\color[rgb]{0,0,0}\makebox(0,0)[lb]{\smash{$\widetilde{x}_3$}}}%
    \put(0.62618908,0.07009291){\color[rgb]{0,0,0}\makebox(0,0)[lb]{\smash{$\widetilde{x}_2$}}}%
    \put(0.71515686,0.06948149){\color[rgb]{0,0,0}\makebox(0,0)[lb]{\smash{$\widetilde{x}_4$}}}%
    \put(0.85594025,0.13592739){\color[rgb]{0,0,0}\makebox(0,0)[lb]{\smash{$\widetilde{x}_6$}}}%
    \put(0.85452673,0.06722728){\color[rgb]{0,0,0}\makebox(0,0)[lb]{\smash{$\widetilde{x}_5$}}}%
    \put(0.94855122,0.13200636){\color[rgb]{0,0,0}\makebox(0,0)[lb]{\smash{$\widetilde{x}_7$}}}%
    \put(1.02412671,0.25537617){\color[rgb]{0,0,0}\makebox(0,0)[lb]{\smash{$\widetilde{x}_1=x_1$}}}%
    \put(1.02616389,0.19808871){\color[rgb]{0,0,0}\makebox(0,0)[lb]{\smash{$\widetilde{x}_7=x_7$}}}%
    \put(0,0){\includegraphics[width=\unitlength,page=2]{tempt.pdf}}%
    \put(0.13328965,0.24468156){\color[rgb]{0,0,0}\makebox(0,0)[lb]{\smash{$x_2$}}}%
    \put(0,0){\includegraphics[width=\unitlength,page=3]{tempt.pdf}}%
    \put(0.11937905,0.37204769){\color[rgb]{0,0,0}\makebox(0,0)[lb]{\smash{$x_3$}}}%
    \put(0,0){\includegraphics[width=\unitlength,page=4]{tempt.pdf}}%
    \put(0.30875014,0.2463911){\color[rgb]{0,0,0}\makebox(0,0)[lb]{\smash{$x_5$}}}%
    \put(0.30846794,0.37204769){\color[rgb]{0,0,0}\makebox(0,0)[lb]{\smash{$x_6$}}}%
    \put(0,0){\includegraphics[width=\unitlength,page=5]{tempt.pdf}}%
    \put(0.22701846,0.35238748){\color[rgb]{0,0,0}\makebox(0,0)[lb]{\smash{$x_4$}}}%
    \put(0,0){\includegraphics[width=\unitlength,page=6]{tempt.pdf}}%
    \put(0.37914744,0.35077776){\color[rgb]{0,0,0}\makebox(0,0)[lb]{\smash{$x_7$}}}%
    \put(0,0){\includegraphics[width=\unitlength,page=7]{tempt.pdf}}%
    \put(0.04894744,0.34795553){\color[rgb]{0,0,0}\makebox(0,0)[lb]{\smash{$x_1$}}}%
    \put(0,0){\includegraphics[width=\unitlength,page=8]{tempt.pdf}}%
    \put(0.14029251,0.06699611){\color[rgb]{0,0,0}\makebox(0,0)[lb]{\smash{$\widetilde{x}_2$}}}%
    \put(0,0){\includegraphics[width=\unitlength,page=9]{tempt.pdf}}%
    \put(0.11795865,0.18828354){\color[rgb]{0,0,0}\makebox(0,0)[lb]{\smash{$\widetilde{x}_3$}}}%
    \put(0,0){\includegraphics[width=\unitlength,page=10]{tempt.pdf}}%
    \put(0.30367743,0.06456465){\color[rgb]{0,0,0}\makebox(0,0)[lb]{\smash{$\widetilde{x}_5$}}}%
    \put(0.31442871,0.18958613){\color[rgb]{0,0,0}\makebox(0,0)[lb]{\smash{$\widetilde{x}_6$}}}%
    \put(0,0){\includegraphics[width=\unitlength,page=11]{tempt.pdf}}%
    \put(0.22849092,0.0789797){\color[rgb]{0,0,0}\makebox(0,0)[lb]{\smash{$\widetilde{x}_4$}}}%
    \put(0,0){\includegraphics[width=\unitlength,page=12]{tempt.pdf}}%
    \put(0.38120054,0.07835244){\color[rgb]{0,0,0}\makebox(0,0)[lb]{\smash{$\widetilde{x}_7$}}}%
    \put(0,0){\includegraphics[width=\unitlength,page=13]{tempt.pdf}}%
    \put(0.05021899,0.08117466){\color[rgb]{0,0,0}\makebox(0,0)[lb]{\smash{$\widetilde{x}_1$}}}%
    \put(0,0){\includegraphics[width=\unitlength,page=14]{tempt.pdf}}%
    \put(0.47882408,0.21864844){\color[rgb]{0,0,0}\makebox(0,0)[lb]{\smash{$=$}}}%
    \put(0.67168888,0.29633328){\color[rgb]{0,0,0}\makebox(0,0)[lb]{\smash{}}}%
  \end{picture}%
\endgroup%

%% file: ocal.pdf_tex
\begingroup%
  \makeatletter%
  \providecommand\color[2][]{%
    \errmessage{(Inkscape) Color is used for the text in Inkscape, but the package 'color.sty' is not loaded}%
    \renewcommand\color[2][]{}%
  }%
  \providecommand\transparent[1]{%
    \errmessage{(Inkscape) Transparency is used (non-zero) for the text in Inkscape, but the package 'transparent.sty' is not loaded}%
    \renewcommand\transparent[1]{}%
  }%
  \providecommand\rotatebox[2]{#2}%
  \ifx\svgwidth\undefined%
    \setlength{\unitlength}{300.15489713bp}%
    \ifx\svgscale\undefined%
      \relax%
    \else%
      \setlength{\unitlength}{\unitlength * \real{\svgscale}}%
    \fi%
  \else%
    \setlength{\unitlength}{\svgwidth}%
  \fi%
  \global\let\svgwidth\undefined%
  \global\let\svgscale\undefined%
  \makeatother%
  \begin{picture}(1,0.40898245)%
    \put(0,0){\includegraphics[width=\unitlength,page=1]{ocal.pdf}}%
    \put(0.15315104,0.1945221){\color[rgb]{0,0,0}\makebox(0,0)[lb]{\smash{$\cdots$}}}%
    \put(0,0){\includegraphics[width=\unitlength,page=2]{ocal.pdf}}%
    \put(0.08149252,0.34636102){\color[rgb]{0,0,0}\makebox(0,0)[lb]{\smash{$x^i_1$}}}%
    \put(0,0){\includegraphics[width=\unitlength,page=3]{ocal.pdf}}%
    \put(0.2418794,0.3585275){\color[rgb]{0,0,0}\makebox(0,0)[lb]{\smash{$x^i_{3k-2}$}}}%
    \put(0,0){\includegraphics[width=\unitlength,page=4]{ocal.pdf}}%
    \put(0.3487766,0.36288959){\color[rgb]{0,0,0}\makebox(0,0)[lb]{\smash{$x^i_{3k}$}}}%
    \put(0.32440816,0.23438457){\color[rgb]{0,0,0}\makebox(0,0)[lb]{\smash{$x^i_{3k-1}$}}}%
    \put(0,0){\includegraphics[width=\unitlength,page=5]{ocal.pdf}}%
    \put(0.4672139,0.37943126){\color[rgb]{0,0,0}\makebox(0,0)[lb]{\smash{$x^i_{3k+1}$}}}%
    \put(0,0){\includegraphics[width=\unitlength,page=6]{ocal.pdf}}%
    \put(0.5873201,0.35832042){\color[rgb]{0,0,0}\makebox(0,0)[lb]{\smash{$x^i_{3k+3}$}}}%
    \put(0.62203647,0.23156921){\color[rgb]{0,0,0}\makebox(0,0)[lb]{\smash{$x^i_{3k+2}$}}}%
    \put(0,0){\includegraphics[width=\unitlength,page=7]{ocal.pdf}}%
    \put(0.80348201,0.1945221){\color[rgb]{0,0,0}\makebox(0,0)[lb]{\smash{$\cdots$}}}%
    \put(0,0){\includegraphics[width=\unitlength,page=8]{ocal.pdf}}%
    \put(0.70684294,0.34958941){\color[rgb]{0,0,0}\makebox(0,0)[lb]{\smash{$x^i_{3k+4}$}}}%
    \put(0,0){\includegraphics[width=\unitlength,page=9]{ocal.pdf}}%
    \put(0.89257891,0.34492068){\color[rgb]{0,0,0}\makebox(0,0)[lb]{\smash{$x^i_{3m+1}$}}}%
    \put(0,0){\includegraphics[width=\unitlength,page=10]{ocal.pdf}}%
    \put(0.08701352,0.0406681){\color[rgb]{0,0,0}\makebox(0,0)[lb]{\smash{$x^{i+1}_1$}}}%
    \put(0,0){\includegraphics[width=\unitlength,page=11]{ocal.pdf}}%
    \put(0.24987527,0.04940781){\color[rgb]{0,0,0}\makebox(0,0)[lb]{\smash{$x^{i+1}_{3k-2}$}}}%
    \put(0,0){\includegraphics[width=\unitlength,page=12]{ocal.pdf}}%
    \put(0.35943792,0.18460754){\color[rgb]{0,0,0}\makebox(0,0)[lb]{\smash{$x^{i+1}_{3k}$}}}%
    \put(0,0){\includegraphics[width=\unitlength,page=13]{ocal.pdf}}%
    \put(0.46645241,0.01510196){\color[rgb]{0,0,0}\makebox(0,0)[lb]{\smash{$x^{i+1}_{3k+1}$}}}%
    \put(0,0){\includegraphics[width=\unitlength,page=14]{ocal.pdf}}%
    \put(0.61219731,0.17977288){\color[rgb]{0,0,0}\makebox(0,0)[lb]{\smash{$x^{i+1}_{3k+3}$}}}%
    \put(0,0){\includegraphics[width=\unitlength,page=15]{ocal.pdf}}%
    \put(0.71483881,0.0458003){\color[rgb]{0,0,0}\makebox(0,0)[lb]{\smash{$x^{i+1}_{3k+4}$}}}%
    \put(0,0){\includegraphics[width=\unitlength,page=16]{ocal.pdf}}%
    \put(0.90324007,0.03618173){\color[rgb]{0,0,0}\makebox(0,0)[lb]{\smash{$x^{i+1}_{3m+1}$}}}%
    \put(0.35144189,0.03486845){\color[rgb]{0,0,0}\makebox(0,0)[lb]{\smash{$x^{i+1}_{3k-1}$}}}%
    \put(0.60083599,0.0431429){\color[rgb]{0,0,0}\makebox(0,0)[lb]{\smash{$x^{i+1}_{3k+2}$}}}%
    \put(0,0){\includegraphics[width=\unitlength,page=17]{ocal.pdf}}%
    \put(-0.04011298,0.27770953){\color[rgb]{0,0,0}\makebox(0,0)[lb]{\smash{$\mathbf{x}^i$}}}%
    \put(0,0){\includegraphics[width=\unitlength,page=18]{ocal.pdf}}%
    \put(-0.07578658,0.09980814){\color[rgb]{0,0,0}\makebox(0,0)[lb]{\smash{$\mathbf{x}^{i+1}$}}}%
  \end{picture}%
\endgroup%

%% file: octahedron.pdf_tex
\begingroup%
  \makeatletter%
  \providecommand\color[2][]{%
    \errmessage{(Inkscape) Color is used for the text in Inkscape, but the package 'color.sty' is not loaded}%
    \renewcommand\color[2][]{}%
  }%
  \providecommand\transparent[1]{%
    \errmessage{(Inkscape) Transparency is used (non-zero) for the text in Inkscape, but the package 'transparent.sty' is not loaded}%
    \renewcommand\transparent[1]{}%
  }%
  \providecommand\rotatebox[2]{#2}%
  \ifx\svgwidth\undefined%
    \setlength{\unitlength}{413.31587183bp}%
    \ifx\svgscale\undefined%
      \relax%
    \else%
      \setlength{\unitlength}{\unitlength * \real{\svgscale}}%
    \fi%
  \else%
    \setlength{\unitlength}{\svgwidth}%
  \fi%
  \global\let\svgwidth\undefined%
  \global\let\svgscale\undefined%
  \makeatother%
  \begin{picture}(1,0.50751446)%
    \put(0,0){\includegraphics[width=\unitlength,page=1]{octahedron.pdf}}%
    \put(0.31372998,0.39091355){\color[rgb]{0,0,0}\makebox(0,0)[lb]{\smash{$x_3$}}}%
    \put(0,0){\includegraphics[width=\unitlength,page=2]{octahedron.pdf}}%
    \put(0.24350631,0.31463463){\color[rgb]{0,0,0}\makebox(0,0)[lb]{\smash{$x_4$}}}%
    \put(0.0776288,0.27218509){\color[rgb]{0,0,0}\makebox(0,0)[lb]{\smash{$x_1$}}}%
    \put(0,0){\includegraphics[width=\unitlength,page=3]{octahedron.pdf}}%
    \put(0.18738209,0.13426259){\color[rgb]{0,0,0}\makebox(0,0)[lb]{\smash{$\widetilde{x}_3$}}}%
    \put(0,0){\includegraphics[width=\unitlength,page=4]{octahedron.pdf}}%
    \put(0.2121496,0.15348522){\color[rgb]{0,0,0}\makebox(0,0)[lb]{\smash{$x_5$}}}%
    \put(0.29424709,0.15844168){\color[rgb]{0,0,0}\makebox(0,0)[lb]{\smash{$x_6$}}}%
    \put(0.40750082,0.27257376){\color[rgb]{0,0,0}\makebox(0,0)[lb]{\smash{$x_7$}}}%
    \put(0,0){\includegraphics[width=\unitlength,page=5]{octahedron.pdf}}%
    \put(0.21476392,0.37860528){\color[rgb]{0,0,0}\makebox(0,0)[lb]{\smash{$x_3$}}}%
    \put(0,0){\includegraphics[width=\unitlength,page=6]{octahedron.pdf}}%
    \put(0.31146564,0.13527987){\color[rgb]{0,0,0}\makebox(0,0)[lb]{\smash{$x_5$}}}%
    \put(0,0){\includegraphics[width=\unitlength,page=7]{octahedron.pdf}}%
    \put(0.2223897,0.2407813){\color[rgb]{0,0,0}\makebox(0,0)[lb]{\smash{$y_2$}}}%
    \put(0.18303147,0.39084149){\color[rgb]{0,0,0}\makebox(0,0)[lb]{\smash{$x_2$}}}%
    \put(0,0){\includegraphics[width=\unitlength,page=8]{octahedron.pdf}}%
    \put(0.29743184,0.36691828){\color[rgb]{0,0,0}\makebox(0,0)[lb]{\smash{$\widetilde{x}_5$}}}%
    \put(0.14771496,0.01042905){\color[rgb]{0,0,0}\makebox(0,0)[lb]{\smash{(a) crossing for $R$}}}%
    \put(0,0){\includegraphics[width=\unitlength,page=9]{octahedron.pdf}}%
    \put(0.22782592,0.28898165){\color[rgb]{0,0,0}\makebox(0,0)[lb]{\smash{$y_1$}}}%
    \put(0,0){\includegraphics[width=\unitlength,page=10]{octahedron.pdf}}%
    \put(0.25359833,0.21117272){\color[rgb]{0,0,0}\makebox(0,0)[lb]{\smash{$\widetilde{x}_4$}}}%
    \put(0,0){\includegraphics[width=\unitlength,page=11]{octahedron.pdf}}%
    \put(0.71092657,0.38970241){\color[rgb]{0,0,0}\makebox(0,0)[lb]{\smash{$x_5$}}}%
    \put(0,0){\includegraphics[width=\unitlength,page=12]{octahedron.pdf}}%
    \put(0.6407029,0.31342336){\color[rgb]{0,0,0}\makebox(0,0)[lb]{\smash{$x_4$}}}%
    \put(0.4748254,0.27336127){\color[rgb]{0,0,0}\makebox(0,0)[lb]{\smash{$x_1$}}}%
    \put(0,0){\includegraphics[width=\unitlength,page=13]{octahedron.pdf}}%
    \put(0.58457868,0.13305161){\color[rgb]{0,0,0}\makebox(0,0)[lb]{\smash{$x_2$}}}%
    \put(0,0){\includegraphics[width=\unitlength,page=14]{octahedron.pdf}}%
    \put(0.60934625,0.15227396){\color[rgb]{0,0,0}\makebox(0,0)[lb]{\smash{$x_3$}}}%
    \put(0.69144369,0.15723042){\color[rgb]{0,0,0}\makebox(0,0)[lb]{\smash{$x_2$}}}%
    \put(0.80469742,0.27136262){\color[rgb]{0,0,0}\makebox(0,0)[lb]{\smash{$x_7$}}}%
    \put(0,0){\includegraphics[width=\unitlength,page=15]{octahedron.pdf}}%
    \put(0.61076697,0.37381314){\color[rgb]{0,0,0}\makebox(0,0)[lb]{\smash{$\widetilde{x}_2$}}}%
    \put(0,0){\includegraphics[width=\unitlength,page=16]{octahedron.pdf}}%
    \put(0.71343688,0.12929396){\color[rgb]{0,0,0}\makebox(0,0)[lb]{\smash{$\widetilde{x}_6$}}}%
    \put(0,0){\includegraphics[width=\unitlength,page=17]{octahedron.pdf}}%
    \put(0.6582976,0.23718283){\color[rgb]{0,0,0}\makebox(0,0)[lb]{\smash{$y_2$}}}%
    \put(0.58022806,0.38963034){\color[rgb]{0,0,0}\makebox(0,0)[lb]{\smash{$x_6$}}}%
    \put(0,0){\includegraphics[width=\unitlength,page=18]{octahedron.pdf}}%
    \put(0.69462844,0.36957827){\color[rgb]{0,0,0}\makebox(0,0)[lb]{\smash{$x_6$}}}%
    \put(0.54491155,0.01308892){\color[rgb]{0,0,0}\makebox(0,0)[lb]{\smash{(b) crossing for $R^{-1}$}}}%
    \put(0,0){\includegraphics[width=\unitlength,page=19]{octahedron.pdf}}%
    \put(0.65539472,0.28657672){\color[rgb]{0,0,0}\makebox(0,0)[lb]{\smash{$y_1$}}}%
    \put(0,0){\includegraphics[width=\unitlength,page=20]{octahedron.pdf}}%
    \put(0.65079492,0.20996158){\color[rgb]{0,0,0}\makebox(0,0)[lb]{\smash{$\widetilde{x}_4$}}}%
    \put(0,0){\includegraphics[width=\unitlength,page=21]{octahedron.pdf}}%
    \put(0.88402287,0.43514389){\color[rgb]{0,0,0}\makebox(0,0)[lb]{\smash{$-1$}}}%
    \put(0.83268396,0.49002912){\color[rgb]{0,0,0}\makebox(0,0)[lb]{\smash{$1$}}}%
    \put(0.76941269,0.43238893){\color[rgb]{0,0,0}\makebox(0,0)[lb]{\smash{$1$}}}%
    \put(0.81475931,0.37121601){\color[rgb]{0,0,0}\makebox(0,0)[lb]{\smash{$-1$}}}%
    \put(0,0){\includegraphics[width=\unitlength,page=22]{octahedron.pdf}}%
    \put(0.81958263,0.43299131){\color[rgb]{0,0,0}\makebox(0,0)[lb]{\smash{$-1$}}}%
    \put(0,0){\includegraphics[width=\unitlength,page=23]{octahedron.pdf}}%
    \put(0.88526131,0.09291989){\color[rgb]{0,0,0}\makebox(0,0)[lb]{\smash{$-1$}}}%
    \put(0.81808913,0.14864914){\color[rgb]{0,0,0}\makebox(0,0)[lb]{\smash{$-1$}}}%
    \put(0.76896317,0.09403605){\color[rgb]{0,0,0}\makebox(0,0)[lb]{\smash{$1$}}}%
    \put(0.82761114,0.028992){\color[rgb]{0,0,0}\makebox(0,0)[lb]{\smash{$1$}}}%
    \put(0,0){\includegraphics[width=\unitlength,page=24]{octahedron.pdf}}%
    \put(0.81913311,0.09146668){\color[rgb]{0,0,0}\makebox(0,0)[lb]{\smash{$-1$}}}%
    \put(0,0){\includegraphics[width=\unitlength,page=25]{octahedron.pdf}}%
    \put(0.12591067,0.43226606){\color[rgb]{0,0,0}\makebox(0,0)[lb]{\smash{$-1$}}}%
    \put(0.05873849,0.48799532){\color[rgb]{0,0,0}\makebox(0,0)[lb]{\smash{$-1$}}}%
    \put(0.00961251,0.43338223){\color[rgb]{0,0,0}\makebox(0,0)[lb]{\smash{$1$}}}%
    \put(0.0682605,0.36833806){\color[rgb]{0,0,0}\makebox(0,0)[lb]{\smash{$1$}}}%
    \put(0,0){\includegraphics[width=\unitlength,page=26]{octahedron.pdf}}%
    \put(0.05978248,0.43081285){\color[rgb]{0,0,0}\makebox(0,0)[lb]{\smash{$-1$}}}%
    \put(0,0){\includegraphics[width=\unitlength,page=27]{octahedron.pdf}}%
    \put(0.12087281,0.09786241){\color[rgb]{0,0,0}\makebox(0,0)[lb]{\smash{$-1$}}}%
    \put(0.0695339,0.15274757){\color[rgb]{0,0,0}\makebox(0,0)[lb]{\smash{$1$}}}%
    \put(0.0062626,0.09510744){\color[rgb]{0,0,0}\makebox(0,0)[lb]{\smash{$1$}}}%
    \put(0.05160925,0.03393452){\color[rgb]{0,0,0}\makebox(0,0)[lb]{\smash{$-1$}}}%
    \put(0,0){\includegraphics[width=\unitlength,page=28]{octahedron.pdf}}%
    \put(0.05643257,0.09570982){\color[rgb]{0,0,0}\makebox(0,0)[lb]{\smash{$-1$}}}%
    \put(0.25880063,0.46367198){\color[rgb]{0,0,0}\makebox(0,0)[lb]{\smash{$0$}}}%
    \put(0.25732179,0.06324385){\color[rgb]{0,0,0}\makebox(0,0)[lb]{\smash{$1$}}}%
    \put(0.08965645,0.21362526){\color[rgb]{0,0,0}\makebox(0,0)[lb]{\smash{$4$}}}%
    \put(0.09248011,0.31455807){\color[rgb]{0,0,0}\makebox(0,0)[lb]{\smash{$3$}}}%
    \put(0.40421034,0.31117851){\color[rgb]{0,0,0}\makebox(0,0)[lb]{\smash{$5$}}}%
    \put(0.40252168,0.21496293){\color[rgb]{0,0,0}\makebox(0,0)[lb]{\smash{$2$}}}%
    \put(0.65590827,0.46268955){\color[rgb]{0,0,0}\makebox(0,0)[lb]{\smash{$0$}}}%
    \put(0.65442942,0.06226142){\color[rgb]{0,0,0}\makebox(0,0)[lb]{\smash{$1$}}}%
    \put(0.80262918,0.31176896){\color[rgb]{0,0,0}\makebox(0,0)[lb]{\smash{$2$}}}%
    \put(0.48690142,0.21360234){\color[rgb]{0,0,0}\makebox(0,0)[lb]{\smash{$3$}}}%
    \put(0.48719174,0.31415617){\color[rgb]{0,0,0}\makebox(0,0)[lb]{\smash{$4$}}}%
    \put(0.8017259,0.21389249){\color[rgb]{0,0,0}\makebox(0,0)[lb]{\smash{$5$}}}%
  \end{picture}%
\endgroup%

%% file: boundarycocycle.pdf_tex
\begingroup%
  \makeatletter%
  \providecommand\color[2][]{%
    \errmessage{(Inkscape) Color is used for the text in Inkscape, but the package 'color.sty' is not loaded}%
    \renewcommand\color[2][]{}%
  }%
  \providecommand\transparent[1]{%
    \errmessage{(Inkscape) Transparency is used (non-zero) for the text in Inkscape, but the package 'transparent.sty' is not loaded}%
    \renewcommand\transparent[1]{}%
  }%
  \providecommand\rotatebox[2]{#2}%
  \ifx\svgwidth\undefined%
    \setlength{\unitlength}{277.90185162bp}%
    \ifx\svgscale\undefined%
      \relax%
    \else%
      \setlength{\unitlength}{\unitlength * \real{\svgscale}}%
    \fi%
  \else%
    \setlength{\unitlength}{\svgwidth}%
  \fi%
  \global\let\svgwidth\undefined%
  \global\let\svgscale\undefined%
  \makeatother%
  \begin{picture}(1,0.34687662)%
    \put(0,0){\includegraphics[width=\unitlength,page=1]{boundarycocycle.pdf}}%
    \put(0.51459886,0.17072928){\color[rgb]{0,0,0}\makebox(0,0)[lb]{\smash{$\cdots$}}}%
    \put(0,0){\includegraphics[width=\unitlength,page=2]{boundarycocycle.pdf}}%
    \put(0.03876602,0.33180294){\color[rgb]{0,0,0}\makebox(0,0)[lb]{\smash{$\mu$}}}%
    \put(0,0){\includegraphics[width=\unitlength,page=3]{boundarycocycle.pdf}}%
    \put(0.95788113,0.22454703){\color[rgb]{0,0,0}\makebox(0,0)[lb]{\smash{: $1$}}}%
    \put(0,0){\includegraphics[width=\unitlength,page=4]{boundarycocycle.pdf}}%
    \put(0.95867153,0.19337139){\color[rgb]{0,0,0}\makebox(0,0)[lb]{\smash{: $-1$}}}%
    \put(0,0){\includegraphics[width=\unitlength,page=5]{boundarycocycle.pdf}}%
    \put(0.74881153,0.00569008){\color[rgb]{0,0,0}\makebox(0,0)[lb]{\smash{$\lambda_{bf}$}}}%
    \put(0.75416963,0.43179286){\color[rgb]{0,0,0}\makebox(0,0)[lt]{\begin{minipage}{0.22741842\unitlength}\raggedright \end{minipage}}}%
    \put(0,0){\includegraphics[width=\unitlength,page=6]{boundarycocycle.pdf}}%
  \end{picture}%
\endgroup%

%% file: developing.pdf_tex
\begingroup%
  \makeatletter%
  \providecommand\color[2][]{%
    \errmessage{(Inkscape) Color is used for the text in Inkscape, but the package 'color.sty' is not loaded}%
    \renewcommand\color[2][]{}%
  }%
  \providecommand\transparent[1]{%
    \errmessage{(Inkscape) Transparency is used (non-zero) for the text in Inkscape, but the package 'transparent.sty' is not loaded}%
    \renewcommand\transparent[1]{}%
  }%
  \providecommand\rotatebox[2]{#2}%
  \ifx\svgwidth\undefined%
    \setlength{\unitlength}{338.31750825bp}%
    \ifx\svgscale\undefined%
      \relax%
    \else%
      \setlength{\unitlength}{\unitlength * \real{\svgscale}}%
    \fi%
  \else%
    \setlength{\unitlength}{\svgwidth}%
  \fi%
  \global\let\svgwidth\undefined%
  \global\let\svgscale\undefined%
  \makeatother%
  \begin{picture}(1,0.47974566)%
    \put(0,0){\includegraphics[width=\unitlength,page=1]{developing.pdf}}%
    \put(0.20612311,0.4289286){\color[rgb]{0,0,0}\makebox(0,0)[lb]{\smash{$p$}}}%
    \put(0,0){\includegraphics[width=\unitlength,page=2]{developing.pdf}}%
    \put(0.21126876,0.03029349){\color[rgb]{0,0,0}\makebox(0,0)[lb]{\smash{$q$}}}%
    \put(0,0){\includegraphics[width=\unitlength,page=3]{developing.pdf}}%
    \put(0.05000878,0.15995122){\color[rgb]{0,0,0}\makebox(0,0)[lb]{\smash{$K$}}}%
    \put(0,0){\includegraphics[width=\unitlength,page=4]{developing.pdf}}%
    \put(0.11460602,0.28899116){\color[rgb]{0,0,0}\makebox(0,0)[lb]{\smash{$e_j$}}}%
    \put(0.25442161,0.31420554){\color[rgb]{0,0,0}\makebox(0,0)[lb]{\smash{$e_k$}}}%
    \put(0,0){\includegraphics[width=\unitlength,page=5]{developing.pdf}}%
    \put(0.46313304,0.269736){\color[rgb]{0,0,0}\makebox(0,0)[lb]{\smash{$\widetilde{e}_j$}}}%
    \put(0,0){\includegraphics[width=\unitlength,page=6]{developing.pdf}}%
    \put(0.15677424,0.29561083){\color[rgb]{0,0,0}\makebox(0,0)[lb]{\smash{$x$}}}%
    \put(0,0){\includegraphics[width=\unitlength,page=7]{developing.pdf}}%
    \put(0.69688263,0.35781425){\color[rgb]{0,0,0}\makebox(0,0)[lb]{\smash{$\widetilde{x}$}}}%
    \put(0.77132738,0.17830404){\color[rgb]{0,0,0}\makebox(0,0)[lb]{\smash{$\widetilde{y}$}}}%
    \put(0,0){\includegraphics[width=\unitlength,page=8]{developing.pdf}}%
    \put(0.1640207,0.17074964){\color[rgb]{0,0,0}\makebox(0,0)[lb]{\smash{$y$}}}%
    \put(0,0){\includegraphics[width=\unitlength,page=9]{developing.pdf}}%
    \put(0.34556383,0.36790432){\color[rgb]{0,0,0}\makebox(0,0)[lb]{\smash{$g$}}}%
    \put(0,0){\includegraphics[width=\unitlength,page=10]{developing.pdf}}%
    \put(0.57334124,0.41207309){\color[rgb]{0,0,0}\makebox(0,0)[lb]{\smash{$\widetilde{p}$}}}%
    \put(0,0){\includegraphics[width=\unitlength,page=11]{developing.pdf}}%
    \put(0.44625176,0.05767494){\color[rgb]{0,0,0}\makebox(0,0)[lb]{\smash{$\widetilde{q}_j$}}}%
    \put(0,0){\includegraphics[width=\unitlength,page=12]{developing.pdf}}%
    \put(0.69419996,0.04422633){\color[rgb]{0,0,0}\makebox(0,0)[lb]{\smash{$\widetilde{q}_k$}}}%
    \put(0,0){\includegraphics[width=\unitlength,page=13]{developing.pdf}}%
    \put(0.87914579,0.27549275){\color[rgb]{0,0,0}\makebox(0,0)[lb]{\smash{$\widetilde{r}=g\cdot\widetilde{r}$}}}%
    \put(0.62901714,0.24217234){\color[rgb]{0,0,0}\makebox(0,0)[lb]{\smash{$\widetilde{e}_k$}}}%
  \end{picture}%
\endgroup%

%% file: developing2.pdf_tex
\begingroup%
  \makeatletter%
  \providecommand\color[2][]{%
    \errmessage{(Inkscape) Color is used for the text in Inkscape, but the package 'color.sty' is not loaded}%
    \renewcommand\color[2][]{}%
  }%
  \providecommand\transparent[1]{%
    \errmessage{(Inkscape) Transparency is used (non-zero) for the text in Inkscape, but the package 'transparent.sty' is not loaded}%
    \renewcommand\transparent[1]{}%
  }%
  \providecommand\rotatebox[2]{#2}%
  \ifx\svgwidth\undefined%
    \setlength{\unitlength}{338.31750825bp}%
    \ifx\svgscale\undefined%
      \relax%
    \else%
      \setlength{\unitlength}{\unitlength * \real{\svgscale}}%
    \fi%
  \else%
    \setlength{\unitlength}{\svgwidth}%
  \fi%
  \global\let\svgwidth\undefined%
  \global\let\svgscale\undefined%
  \makeatother%
  \begin{picture}(1,0.47974566)%
    \put(0,0){\includegraphics[width=\unitlength,page=1]{developing2.pdf}}%
    \put(0.24868668,0.4289286){\color[rgb]{0,0,0}\makebox(0,0)[lb]{\smash{$p$}}}%
    \put(0,0){\includegraphics[width=\unitlength,page=2]{developing2.pdf}}%
    \put(0.25383233,0.03029349){\color[rgb]{0,0,0}\makebox(0,0)[lb]{\smash{$q$}}}%
    \put(0,0){\includegraphics[width=\unitlength,page=3]{developing2.pdf}}%
    \put(0.09939362,0.16336187){\color[rgb]{0,0,0}\makebox(0,0)[lb]{\smash{$K$}}}%
    \put(0,0){\includegraphics[width=\unitlength,page=4]{developing2.pdf}}%
    \put(0.15716959,0.28899116){\color[rgb]{0,0,0}\makebox(0,0)[lb]{\smash{$e_j$}}}%
    \put(0.30103238,0.30938519){\color[rgb]{0,0,0}\makebox(0,0)[lb]{\smash{$e_k$}}}%
    \put(0,0){\includegraphics[width=\unitlength,page=5]{developing2.pdf}}%
    \put(0.55298946,0.269736){\color[rgb]{0,0,0}\makebox(0,0)[lb]{\smash{$\widetilde{e}_j$}}}%
    \put(0,0){\includegraphics[width=\unitlength,page=6]{developing2.pdf}}%
    \put(0.65168482,0.10892854){\color[rgb]{0,0,0}\makebox(0,0)[lb]{\smash{$\widetilde{x}$}}}%
    \put(0,0){\includegraphics[width=\unitlength,page=7]{developing2.pdf}}%
    \put(0.75509016,0.26831511){\color[rgb]{0,0,0}\makebox(0,0)[lb]{\smash{$\widetilde{e}_k$}}}%
    \put(0.23137872,0.16785561){\color[rgb]{0,0,0}\makebox(0,0)[lb]{\smash{$x$}}}%
    \put(0,0){\includegraphics[width=\unitlength,page=8]{developing2.pdf}}%
    \put(0.37940702,0.35973235){\color[rgb]{0,0,0}\makebox(0,0)[lb]{\smash{$g$}}}%
    \put(0,0){\includegraphics[width=\unitlength,page=9]{developing2.pdf}}%
    \put(0.66638431,0.40827997){\color[rgb]{0,0,0}\makebox(0,0)[lb]{\smash{$\widetilde{p}$}}}%
    \put(0,0){\includegraphics[width=\unitlength,page=10]{developing2.pdf}}%
    \put(0.53674181,0.06093195){\color[rgb]{0,0,0}\makebox(0,0)[lb]{\smash{$\widetilde{q}_j$}}}%
    \put(0,0){\includegraphics[width=\unitlength,page=11]{developing2.pdf}}%
    \put(0.78376254,0.05029785){\color[rgb]{0,0,0}\makebox(0,0)[lb]{\smash{$\widetilde{q}_k$}}}%
    \put(0,0){\includegraphics[width=\unitlength,page=12]{developing2.pdf}}%
  \end{picture}%
\endgroup%

%% file: rule.pdf_tex
\begingroup%
  \makeatletter%
  \providecommand\color[2][]{%
    \errmessage{(Inkscape) Color is used for the text in Inkscape, but the package 'color.sty' is not loaded}%
    \renewcommand\color[2][]{}%
  }%
  \providecommand\transparent[1]{%
    \errmessage{(Inkscape) Transparency is used (non-zero) for the text in Inkscape, but the package 'transparent.sty' is not loaded}%
    \renewcommand\transparent[1]{}%
  }%
  \providecommand\rotatebox[2]{#2}%
  \ifx\svgwidth\undefined%
    \setlength{\unitlength}{283.46456693bp}%
    \ifx\svgscale\undefined%
      \relax%
    \else%
      \setlength{\unitlength}{\unitlength * \real{\svgscale}}%
    \fi%
  \else%
    \setlength{\unitlength}{\svgwidth}%
  \fi%
  \global\let\svgwidth\undefined%
  \global\let\svgscale\undefined%
  \makeatother%
  \begin{picture}(1,0.25670555)%
    \put(0.03046597,0.17137417){\color[rgb]{0,0,0}\makebox(0,0)[lt]{\begin{minipage}{0.12095237\unitlength}\raggedright $V_{j_1}$\end{minipage}}}%
    \put(0,0){\includegraphics[width=\unitlength,page=1]{rule.pdf}}%
    \put(0.13104748,0.08673541){\color[rgb]{0,0,0}\makebox(0,0)[lt]{\begin{minipage}{0.59569049\unitlength}\raggedright $V_{j_2}=\widetilde{\rho}(g_i)^{-1} V_{j_1}$\end{minipage}}}%
    \put(0,0){\includegraphics[width=\unitlength,page=2]{rule.pdf}}%
    \put(0.51920212,0.17129958){\color[rgb]{0,0,0}\makebox(0,0)[lt]{\begin{minipage}{0.12095237\unitlength}\raggedright $i_1$\end{minipage}}}%
    \put(0.6897561,0.29768072){\color[rgb]{0,0,0}\makebox(0,0)[lt]{\begin{minipage}{0.12095237\unitlength}\raggedright $i_2$\end{minipage}}}%
    \put(0.84987211,0.16747953){\color[rgb]{0,0,0}\makebox(0,0)[lt]{\begin{minipage}{0.12095237\unitlength}\raggedright $i_3$\end{minipage}}}%
    \put(0.76418913,0.08448667){\color[rgb]{0,0,0}\makebox(0,0)[lt]{\begin{minipage}{0.50898863\unitlength}\raggedright $H_{i_3} = \widetilde{\rho}(g_{i_2})^{-1}H_{i_1}$\end{minipage}}}%
    \put(0.2676094,0.23796049){\color[rgb]{0,0,0}\makebox(0,0)[lt]{\begin{minipage}{0.12095237\unitlength}\raggedright $i$\end{minipage}}}%
  \end{picture}%
\endgroup%

%% file: rule2.pdf_tex
\begingroup%
  \makeatletter%
  \providecommand\color[2][]{%
    \errmessage{(Inkscape) Color is used for the text in Inkscape, but the package 'color.sty' is not loaded}%
    \renewcommand\color[2][]{}%
  }%
  \providecommand\transparent[1]{%
    \errmessage{(Inkscape) Transparency is used (non-zero) for the text in Inkscape, but the package 'transparent.sty' is not loaded}%
    \renewcommand\transparent[1]{}%
  }%
  \providecommand\rotatebox[2]{#2}%
  \ifx\svgwidth\undefined%
    \setlength{\unitlength}{141.73228346bp}%
    \ifx\svgscale\undefined%
      \relax%
    \else%
      \setlength{\unitlength}{\unitlength * \real{\svgscale}}%
    \fi%
  \else%
    \setlength{\unitlength}{\svgwidth}%
  \fi%
  \global\let\svgwidth\undefined%
  \global\let\svgscale\undefined%
  \makeatother%
  \begin{picture}(1,0.66)%
    \put(0,0){\includegraphics[width=\unitlength,page=1]{rule2.pdf}}%
    \put(0.45947713,0.2890131){\color[rgb]{0,0,0}\makebox(0,0)[lt]{\begin{minipage}{1.19138113\unitlength}\raggedright $j$\end{minipage}}}%
    \put(0.59013502,0.51223169){\color[rgb]{0,0,0}\makebox(0,0)[lt]{\begin{minipage}{0.24190477\unitlength}\raggedright $i$\end{minipage}}}%
    \put(0,0){\includegraphics[width=\unitlength,page=2]{rule2.pdf}}%
    \put(0.75555254,0.43472837){\color[rgb]{0,0,0}\makebox(0,0)[lt]{\begin{minipage}{1.38000202\unitlength}\raggedright regional edge : $\textrm{det}(V_{j}, W)$\end{minipage}}}%
    \put(0,0){\includegraphics[width=\unitlength,page=3]{rule2.pdf}}%
    \put(1.05268912,0.02339006){\color[rgb]{0,0,0}\makebox(0,0)[lb]{\smash{}}}%
    \put(0.44064492,0.61211774){\color[rgb]{0,0,0}\makebox(0,0)[lb]{\smash{over-edge : $\textrm{det}(H_i, W)$}}}%
    \put(0.24995611,0.01051849){\color[rgb]{0,0,0}\makebox(0,0)[lb]{\smash{under-edge : $\textrm{det}(V_j,H_i)$}}}%
  \end{picture}%
\endgroup%

%% file: braid.pdf_tex
\begingroup%
  \makeatletter%
  \providecommand\color[2][]{%
    \errmessage{(Inkscape) Color is used for the text in Inkscape, but the package 'color.sty' is not loaded}%
    \renewcommand\color[2][]{}%
  }%
  \providecommand\transparent[1]{%
    \errmessage{(Inkscape) Transparency is used (non-zero) for the text in Inkscape, but the package 'transparent.sty' is not loaded}%
    \renewcommand\transparent[1]{}%
  }%
  \providecommand\rotatebox[2]{#2}%
  \ifx\svgwidth\undefined%
    \setlength{\unitlength}{141.73228346bp}%
    \ifx\svgscale\undefined%
      \relax%
    \else%
      \setlength{\unitlength}{\unitlength * \real{\svgscale}}%
    \fi%
  \else%
    \setlength{\unitlength}{\svgwidth}%
  \fi%
  \global\let\svgwidth\undefined%
  \global\let\svgscale\undefined%
  \makeatother%
  \begin{picture}(1,1.1)%
    \put(0,0){\includegraphics[width=\unitlength,page=1]{braid.pdf}}%
    \put(0.91108337,0.81281973){\color[rgb]{0,0,0}\makebox(0,0)[lb]{\smash{$\mathbf{x}^{2}$}}}%
    \put(0,0){\includegraphics[width=\unitlength,page=2]{braid.pdf}}%
    \put(0.91614077,0.62744189){\color[rgb]{0,0,0}\makebox(0,0)[lb]{\smash{$\mathbf{x}^{3}$}}}%
    \put(0.90235952,0.43561295){\color[rgb]{0,0,0}\makebox(0,0)[lb]{\smash{$\mathbf{x}^{4}$}}}%
    \put(0.89583589,0.23073691){\color[rgb]{0,0,0}\makebox(0,0)[lb]{\smash{$\mathbf{x}^{5}$}}}%
    \put(0,0){\includegraphics[width=\unitlength,page=3]{braid.pdf}}%
    \put(0.90838947,0.05445826){\color[rgb]{0,0,0}\makebox(0,0)[lb]{\smash{$\mathbf{x}^{6}$}}}%
    \put(0,0){\includegraphics[width=\unitlength,page=4]{braid.pdf}}%
    \put(0.91446438,1.02247049){\color[rgb]{0,0,0}\makebox(0,0)[lb]{\smash{$\mathbf{x}^1$}}}%
    \put(0,0){\includegraphics[width=\unitlength,page=5]{braid.pdf}}%
    \put(0.10390692,1.00129987){\color[rgb]{0,0,0}\makebox(0,0)[lb]{\smash{}}}%
    \put(0.11064626,1.05074818){\color[rgb]{0,0,0}\makebox(0,0)[lb]{\smash{$1$}}}%
    \put(0.30289973,1.05439706){\color[rgb]{0,0,0}\makebox(0,0)[lb]{\smash{$2$}}}%
    \put(0.70365529,1.0600415){\color[rgb]{0,0,0}\makebox(0,0)[lb]{\smash{$3$}}}%
    \put(0.50045529,1.05274375){\color[rgb]{0,0,0}\makebox(0,0)[lb]{\smash{$4$}}}%
    \put(0.73187751,0.72137483){\color[rgb]{0,0,0}\makebox(0,0)[lb]{\smash{$5$}}}%
    \put(-0.03046485,0.52581485){\color[rgb]{0,0,0}\makebox(0,0)[lb]{\smash{$1$}}}%
    \put(0.1843664,0.5520415){\color[rgb]{0,0,0}\makebox(0,0)[lb]{\smash{$2$}}}%
    \put(0.39321084,0.91893039){\color[rgb]{0,0,0}\makebox(0,0)[lb]{\smash{$3$}}}%
    \put(0.41578862,0.54474375){\color[rgb]{0,0,0}\makebox(0,0)[lb]{\smash{$4$}}}%
    \put(0.57947751,0.72701928){\color[rgb]{0,0,0}\makebox(0,0)[lb]{\smash{$5$}}}%
    \put(0.59641084,0.33190817){\color[rgb]{0,0,0}\makebox(0,0)[lb]{\smash{$6$}}}%
    \put(0.80525529,0.52946372){\color[rgb]{0,0,0}\makebox(0,0)[lb]{\smash{$7$}}}%
    \put(0,0){\includegraphics[width=\unitlength,page=6]{braid.pdf}}%
  \end{picture}%
\endgroup%